\documentclass[11pt,leqno]{amsart}
\usepackage{amssymb,graphicx,color} 
\begin{document}
\theoremstyle{plain}
\newtheorem{thm}{Theorem}[section]
\newtheorem{prop}[thm]{Proposition}
\newtheorem{lem}[thm]{Lemma}
\newtheorem{clry}[thm]{Corollary}
\newtheorem{deft}[thm]{Definition}
\newtheorem{hyp}{Assumption}
\newtheorem*{KSU}{Theorem (Kenig, Sj\"ostrand and Uhlmann)}

\theoremstyle{definition}
\newtheorem{rem}[thm]{Remark}
\numberwithin{equation}{section}
\newcommand{\eps}{\varepsilon}
\renewcommand{\d}{\partial}
\newcommand{\dd}{\mathrm{d}}
\newcommand{\e}{\mathrm{e}}
\newcommand{\re}{\mathop{\rm Re} }
\newcommand{\im}{\mathop{\rm Im}}
\newcommand{\ch}{\mathop{\rm ch}}
\newcommand{\R}{\mathbf{R}}
\newcommand{\C}{\mathbf{C}}
\renewcommand{\H}{\mathbf{H}} 
\newcommand{\N}{\mathbf{N}} 
\newcommand{\D}{\mathcal{C}^{\infty}_0} 
\newcommand{\supp}{\mathop{\rm supp}}
\hyphenation{pa-ra-met-ri-zed}
\title[]{Stability estimates for the Calder\'on problem with partial data}
\author[]{Pedro Caro \and David Dos Santos Ferreira \and Alberto Ruiz}
\address{Department of Mathematics and Statistics, University of Helsinki, Helsinki, Finland}
\email{pedro.caro@helsinki.fi}
\address{Institut \'Elie Cartan, Universit\'e de Lorraine, CNRS,  B.P. 70239, F-54506 Vandoeuvre-l\`es-Nancy Cedex, France}
\email{ddsf@math.cnrs.fr}
\address{Departamento de Matem\'aticas, Universidad Aut\'onoma de Madrid, Campus de Cantoblanco, 28049 Madrid, Spain}
\email{alberto.ruiz@uam.es}
\begin{abstract}
   This is a {follow-up} of our previous article \cite{CDSFR} where we proved local stability estimates for a potential in a Schr{\"o}dinger equation on an open bounded set in dimension $n=3$ from the Dirichlet-to-Neumann map with partial data. The region under control was the penumbra delimited by a source of light outside of the convex hull of the open set. These local estimates provided
stability of log-log type corresponding to the uniqueness results in Calder\'on's inverse problem with partial data proved by Kenig, Sj\"ostrand and Uhlmann~\cite{KSU}. In this article, we prove 
the corresponding global estimates in all dimensions higher than three.
The estimates are based on the construction of solutions of the Schr\"odinger equation by complex geometrical optics developed in the anisotropic setting by Dos Santos Ferreira, Kenig, Salo and Uhlmann~\cite{DSFKSU} to solve the Calder\'on problem in certain admissible geometries.
\end{abstract}
\maketitle
\setcounter{tocdepth}{1} 
\tableofcontents
%
%
\begin{section}{Introduction}

The aim of this article is to complete our previous results on the stability of Calder\'on's inverse problem with partial data. In~\cite{CDSFR} we obtained local stability estimates of log-log type on a potential in a Schr\"odinger equation from the Dirichlet-to-Neumann map with partial data in dimension three. Unfortunately, we could only control the potentials in the penumbra region delimited by a source of light located outside of the convex hull of the open set where the potentials were supported. It seems very likely that one could obtain global estimates by a continuation argument but in the gluing process, it might be difficult to avoid {losing} as many logarithms as there are steps, thus yielding very weak stability estimates.

The proof was based on the construction of solutions of the Schr\"odinger equation by complex geometrical optics with nonlinear phases performed by Kenig, Sj\"ostrand and Uhlmann \cite{KSU} and the alternative final argument of the same authors with Dos Santos Ferreira \cite{DSFKSjU} involving the Radon transform. The estimates were thus derived in two steps :
\begin{itemize}
    \item[--] first, controlling the Radon transform with restricted angle-distance data by the Dirichlet-to-Neumann map 
    \item[--] then, controlling a function on some determined subset by this restricted Radon transform.
\end{itemize}
Both steps provided log-stability, which after combination yielded a log-log estimate. The second step was based on the observation that the qualitative microlocal analytic methods used in \cite{DSFKSjU} (and in a slightly more involved way in \cite{KSU}) could actually be transformed into quantitative estimates, and relied on the use of Kashiwara's Watermelon theorem. An observation that we still believe to be of interest. The microlocal nature of the argument (also related to unique continuation) partly explains the reason why one only obtains local estimates. Besides, the restriction to dimension three was justified 
by the fact that in this case, the Radon transform naturally comes in, whereas in higher dimensions, it is replaced by the two-plane transform, and that formulae relating the Radon and the Segal-Bargmann transforms (used in the microlocal arguments) are straightforward. 

In this paper, we choose a different path which directly yields global estimates and applies similarly to all dimensions higher than three. It is based on the methods developped in \cite{DSFKSU} (and refined in \cite{DSFKLS}) to tackle the Calder\'on problem on manifolds under certain geometric assumptions, and which also englobe part of the construction done in \cite{KSU}. This observation has already been made by Kenig and Salo \cite{KS} and is the starting point to improve uniqueness results in the inverse problem with partial data and to unify
different approaches on the subject%
\footnote{That is the construction in \cite{KSU} of complex geometrical optics with nonlinear phases in the Euclidean case, the study of limiting Carleman weights by \cite{DSFKSU} in the manifold case and the use of reflection arguments by Isakov \cite{Is} for open sets for which
parts of  the boundary are hyperplanes or hyperspheres.}. Like the predecessor of this article, the proof of the estimates decomposes into two parts involving this time the (attenuated) geodesic ray transform on the hemisphere, both yielding log-stability. The upgrade to global estimates is allowed because the information provided on the ray transform is not restricted to certain points and directions.

Let us now describe the problem under scope and our results. We consider the  Schr\"odinger equation 
with bounded potential $q \in L^{\infty}(\Omega)$ 
\begin{align}
\label{Intro:DirichletSchrod}
     \begin{cases}
            -\Delta u + q u=0  & \text{ in } \Omega \\ u|_{\d \Omega} = f \in H^{\frac{1}{2}}(\d \Omega)
     \end{cases}.
\end{align}
When $0$ is not a Dirichlet eigenvalue of the Schr\"odinger 
operator $-\Delta +q$, the Dirichlet-to-Neumann map may be defined  by
      $$ \Lambda_{q} f = \d_{\nu} u|_{\d\Omega} $$
where $\nu$ is the exterior unit normal of $\Omega$ {and $\partial \Omega$ denotes the boundary of $\Omega$}. This {is} a bounded operator 
      $$ \Lambda_q:H^{\frac{1}{2}}(\d \Omega) \to H^{-\frac{1}{2}}(\d \Omega) $$ 
---in fact a first order pseudodifferential operator when $q$ {and $\partial \Omega$ are} smooth. The Schr\"odinger equation relates to the conductivity equation involved in the Calder\'on problem, at least for smooth enough conductivities, as observed by Sylvester and Uhlmann in {\cite{SyU}}. In combination with  boundary determination obtained by Kohn and Vogelius \cite{KV} this allowed Sylvester and Uhlmann to solve the Calder\'on problem for twice continuously differentiable conductivities in dimension $n \geq 3$.      
      
The formulation of the inverse problem with partial data is whether the Dirichlet-to-Neumann map $\Lambda_q$ measured on subsets of the boundary determines the electric potential $q$ inside $\Omega$.
In dimension higher than three, the first results with partial data were obtained by Bukhgeim and Uhlmann  \cite{BU} but 
required measurements on roughly half of the boundary. 
The first results which require measurements on small subsets
of the boundary (at least for strictly convex open sets $\Omega$) were
obtained by Kenig, Sj\"ostrand and Uhlmann \cite{KSU}. With the refinements and generalization recently obtained by Kenig and Salo \cite{KS}, they are the most precise results so far in dimension ~$n \geq 3$. Let us describe the results of \cite{KSU} in more details. For that purpose, we introduce the front
and back sets
\begin{align*}
      F(x_0) &= \big\{x \in \d\Omega : \langle x-x_0,\nu \rangle \leq 0 \big\} \\
      B(x_0) &= \big\{x \in \d\Omega : \langle x-x_0,\nu \rangle \geq 0 \big\} 
\end{align*}
of $\bar{\Omega}$ with respect to a source $x_0 \in \R^n \setminus \mathop{\rm ch}(\bar{\Omega})$ outside the convex hull of 
$\bar{\Omega}$ (see figure \ref{fig:FrontBack}). Recall that $\nu$ is the exterior unit normal. The main result of \cite{KSU} reads as follows.
\begin{KSU}
      Let $\Omega$ be a bounded open set in $\R^n, n \ge 3$ with smooth boundary, let $x_0 \notin \ch(\bar{\Omega})$ 
      and consider two open neighbourhoods $\tilde{F},\tilde{B}$ respectively of the front and back sets $F(x_0)$ and $B(x_0)$
      of $\bar{\Omega}$ with respect to $x_0$. Let $q_1,q_2$ be two bounded potentials on $\Omega$, suppose that $0$ is 
      neither a Dirichlet eigenvalue of the Schr\"odinger operator $-\Delta+q_1$ nor of $-\Delta+q_2$, and that for all
      $f \in H^{1/2}(\d\Omega)$ supported in $\tilde{B}$, the two Dirichlet-to-Neumann maps coincide on $\tilde{F}$
           $$ \Lambda_{q_1}f(x) = \Lambda_{q_2}f(x), \quad x \in \tilde{F} $$
      then the two potentials agree $q_1=q_2$.
\end{KSU}
\begin{center}
\begin{figure}[ht]
     \input{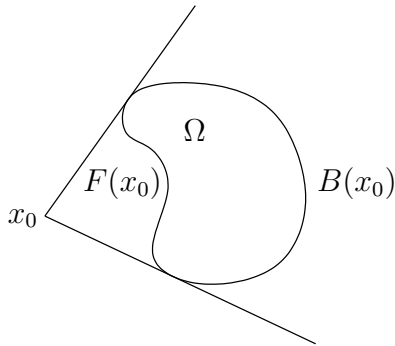}
     \caption{Front and back sets of an open set}
     \label{fig:FrontBack}
\end{figure}
\end{center}

The main goal of this article is to derive the corresponding stability estimates. Stability estimates for the conductivity
inverse problem go back to Alessandrini's article \cite{A}. 
For the inverse problem with partial data, stability estimates corresponding to the results of Bukhgeim and Uhlmann were derived by Heck and Wang \cite{HW},
and in the presence of a magnetic field by Tzou \cite{Tz}. Recently, Caro and Salo proved in \cite{CSa} stability estimates for the Calder\'on problem in certain anisotropic media following an approach quite similar to the one presented here.
Our previous article \cite{CDSFR} was, to the best of our knowledge, the first attempt to obtain stability estimates corresponding to the results of Kenig, Sj\"ostrand and Uhlmann, and only lacked its goal by the local nature of the estimates. {To simplify the exposition of the derivation of our stability estimates, we only consider the case where the Dirichlet-to-Neumann maps are measured on the front set of $\Omega$ but with test functions fully supported in the boundary (instead of the back front set). We believe the extension to the case considered by Kenig, Sj\"ostrand and Uhlmann to be a purely technical matter}%
\footnote{This requires in particular the constuction of exponentially small corrections by a reflection argument to obtain the correct support properties of the traces of the complex geometrical optics solutions on the boundary with the complexified parameter $s \in \C$ considered in this paper, along the lines of the construction done in \cite{KSU}.}.


Let us begin by defining a partial Dirichlet-to-Neumann map; 
let $\tilde{F}$ be an open neighbourhood of the front set $F(x_0)$. Let $\chi_{\tilde{F}}$ be a cutoff function supported in $\tilde{F}$ which equals $1$ on a neighbourhood of $F(x_0)$. We consider the following partial Dirichlet-to-Neumann map
     $$\Lambda_{q}^{\sharp} =  \chi_{\tilde{F}}\Lambda_q. $$
{To obtain reasonable stability results, one needs \textit{a priori} bounds on the potentials in the Shr\"odinger equation. The class of potentials we are considering is the following: }
\begin{align*}
     \mathcal{Q}(M,\sigma) = \big\{ q \in L^{\infty} \cap H^{\sigma}(\R^n) : \supp q \subset \bar{\Omega}, \,\|q\|_{L^{\infty}}+\|q\|_{H^{\sigma}} \leq M \big\}.
\end{align*}
{Our stability result is as follows.}
\begin{thm}
\label{Intro:MainThm}
     Let $\Omega$ be a bounded open set in $\R^n, n \ge 3$ with smooth boundary, let $x_0 \notin \ch(\bar{\Omega})$ 
      and consider an open neighbourhood $\tilde{F}$ of the front set $F(x_0)$
      of $\bar{\Omega}$ with respect to $x_0$. Let $\sigma \in (0,1/2)$ and $M>0$, there exists a constant $C>0$ (depending on $\Omega, \tilde{F}$, $n$, $\sigma$ and $M$) such that for all  bounded potentials $q_1,q_2 \in \mathcal{Q}(M,\sigma)$ such that $0$ is 
      neither a Dirichlet eigenvalue of the Schr\"odinger operator $-\Delta+q_1$ nor of $-\Delta+q_2$, the estimate
     \begin{align*}
          \|q_1-q_2\|_{L^2(\Omega)} \leq C \big|\log \big|\log \|\Lambda^{\sharp}_{q_1}-\Lambda^{\sharp}_{q_2}\|\big|\big|^{-\frac{2\sigma}{3+3\sigma}}
     \end{align*}
     holds true.
\end{thm}

The conductivity equation in  $\Omega$ reads as follows
\begin{align}
\label{Intro:DirichletCond}
     \begin{cases}
           \mathop{\rm div}(\gamma \nabla u) =0  & \text{ in } \Omega \\ u |_{\d \Omega} = f \in H^{\frac{1}{2}}(\d \Omega)
     \end{cases}
\end{align}
where $\gamma$ is a positive function of class $\mathcal{C}^2$ on $\overline{\Omega}$. The Dirichlet-to-Neumann map is defined as before
     $$ \Lambda_{\gamma} f = \gamma \d_{\nu} u|_{\d\Omega} $$
where $\d_{\nu}$ denotes the exterior normal derivative of $u$. 
With a slight abuse of notations, we use the convention that whenever the subscript contains the letter $q$, the notation
refers to the Dirichlet-to-Neumann map related to the Schr\"odinger equation \eqref{Intro:DirichletSchrod}, while
if it contains the letter $\gamma$ it refers to the map related to the conductivity equation \eqref{Intro:DirichletCond}.
The inverse problem formulated%
\footnote{ In fact in its initial formulation, the problem concerns only measurable conductivities bounded from above and below, and remains so far open in dimensions higher than three.} by Calder\'on \cite{C} is whether it is possible to determine $\gamma$ from~$\Lambda_{\gamma}$.
The class of admissible conductivities that we will consider is:
\begin{align*}
     \mathcal{G}(M,\sigma) = \big\{ \gamma \in W^{2,\infty}\cap H^{2+\sigma}(\R^n) : \supp \gamma \subset \bar{\Omega}, \|\gamma\|_{W^{2,\infty}} {+ \| \gamma \|_{H^{2+\sigma}}} \leq M \big\}.
\end{align*}
\begin{clry}
\label{Intro:MainClry}
      Let $\Omega$ be a bounded open set in $\R^n, n \ge 3$ with smooth boundary, let $x_0 \notin \ch(\bar{\Omega})$ 
      and consider an open neighbourhood $\tilde{F}$ of the front set $F(x_0)$
      of $\bar{\Omega}$ with respect to $x_0$. Let $\sigma \in (0,1]$ and $M>0$, there exists a constant $C>0$ (depending on $\Omega, \tilde{F}$, $n$, $\sigma$ and $M$) such that for all  bounded conductivities $\gamma_1,\gamma_2 \in \mathcal{G}(M,\sigma)$ the estimate
	\begin{align*}
		\| \gamma_1 - \gamma_2 \|_{H^1(\Omega)} \leq C \bigg|\log \Big|\log \Big(& \big\|\Lambda^{\sharp}_{\gamma_1}-\Lambda^{\sharp}_{\gamma_2} \big\|
		+ \| \gamma_1 - \gamma_2 \|_{L^\infty(\partial \Omega)}\\
		 & + \| \nabla \gamma_1 - \nabla \gamma_2 \|_{L^\infty(\partial \Omega)} \Big)\Big| \bigg|^{-\frac{2\sigma}{3+3\sigma}}
	\end{align*}
holds true.
\end{clry}
\begin{rem} As a consequence of this corollary, one can also get control of $\| \gamma_1 - \gamma_2 \|_{L^\infty}$. Indeed, it is enough to interpolate between $H^1$ and $W^{1,\infty}$ to get control of $\| \gamma_1 - \gamma_2 \|_{W^{1,p}}$ with $p > n$ and then use Sobolev embeddings.
\end{rem}

In the second section, we study the attenuated geodesic ray transform on a hemisphere; the results are classical but given the plain context we choose a very explicit and pedestrian approach. This could also be thought as an introduction on the subject of geodesic ray transforms on manifolds with the example of the hypersphere. In the third section, we construct solutions to the Schr\"odinger equation by complex geometrical optics, {using the approach developped in the anisotropic setting in \cite{DSFKSU} and in \cite{DSFKLS}.}
In section four, we establish the stability estimates of Theorem \ref{Intro:MainThm}; {first we show how to extend bounds on the Fourier transform of a function in a given interval to a larger set of frequencies, then we control a mixed Fourier transform in the radial variables, and geodesic ray transform in the angular variable of the potentials by the Dirichlet-to-Neumann map and finish by using the stability estimates derived in section two.}
In the last section, we go back to the conductivity equation (and thus to Calder\'on's problem) and prove Corollary \ref{Intro:MainClry}.

\subsection*{Acknowledgements} David DSF wishes to thank Francis Chung for useful discussions around the Calder\'on problem with partial data.

\subsection*{Notations}  We denote $\H=\{s\in \C : \im s>0\}$ Poincar\'e's upper half plane.
The convention that we are using on the Fourier transform is the following
      $$ \widehat{f}(\lambda) = \int_{-\infty}^{\infty} \e^{- i \ell \lambda} f(\ell) \, \dd \ell. $$
We use the classical Japanese brackets notation 
   $$ \langle z \rangle = \sqrt{1+|z|^2}, \quad z \in \R^n. $$   
{Moreover, if $y \in \R^n$ then $y^\perp$ denotes the hyperplane orthogonal to $y$.}
Finally $(S^d,g_d)$ denotes the hypersphere of dimension $d$ in $\R^{d+1}$ endowed with the canonical metric $g_d$.
\end{section}
%
%
\section{The attenuated geodesic ray transform on a hemisphere}
\label{Xray:sec}

We consider the northern hemisphere (with north pole $e_{d+1}=(0,\dots,0,1)$)
    $$ S_+^d = \big\{ x \in S^d : x_{d+1}>0 \big\} \subset \R^{d+1} $$
the boundary of which is the equator
    $$ \d S^d_+ = \big\{ (x',0) \in S^d: x' \in S^{d-1}\big\} \simeq S^{d-1}. $$
The exterior unit normal to the boundary is everywhere $(0,-1)$ and the
set of inward pointing unit tangent vectors 
\begin{align*}
     \d_+S(S^d_+) = \big\{(x,\xi) \in S^d \times S^d : x_{d+1} = 0, \xi_{d+1}>0, \langle x',\xi' \rangle =0 \big\}
\end{align*}
can be identified with a subset of $\d S_+^d \times S^d_+$. Geodesics starting at the boundary {of $S_+^d$} are half great circles which can be parametrized in the following way
\begin{align*}
     \gamma_{x',\xi}(t) &= \cos t \, (x',0) + \sin t \, \xi = (\cos t \, x'+\sin t \, \xi',\sin t \, \xi_{d+1}) \\ &= \exp_{(x',0)} (t \xi)
\end{align*}
with $\langle x',\xi' \rangle =0$ and $\xi_{d+1}>0$. The attenuated geodesic ray transform of a function is defined as 
\begin{align*}
     T_{\lambda}^+f(x',\xi) = \int_0^{\pi} f(\cos t \, (x',0) + \sin t \, \xi) \, \e^{-\lambda t} \, \dd t
\end{align*}
with $\lambda \geq 0$. Recall that the distance on the sphere is given by
     $$ d_{S^d}(x,y) = \arccos \langle x,y\rangle. $$

\subsection{Caps strictly smaller than hemispheres are simple manifolds}   

Let $\alpha_0 \in ]-1,1[$, we denote  
\begin{align*}
     S^{d}_{>\alpha_0} &= \big\{ x \in S^{d} : x_{d+1} > \alpha_0 \big\} 
\end{align*}
caps on the hypersphere; these are strictly smaller than the northern hemisphere $S^{n-1}_{>\alpha_0} \subsetneq S^d_+$ provided $\alpha_0>0$.
\begin{deft}
      A compact manifold with boundary $(M,g)$ is said to be \emph{simple} if its boundary $\d M$ is stricly convex (that is its second fundamental form is definite positive), and if the exponential map $\exp_x : U_x \to M$ defined on its maximal set of definition is a diffeomorphism for all $x \in M$. 
\end{deft}  
The importance of the notion of simple manifolds is partly justified by the fact that the geodesic ray transform is known to be injective on such manifolds. The boundary
     $$ \d S^d_{>\alpha_0} =  \big\{ x \in S^{d} : x_{d+1} = \alpha_0 \big\} $$
of a cap is stricly convex if and only if the cap is strictly smaller than the hemisphere. Indeed $x \mapsto (\alpha_0-x_{d+1})$ is a boundary defining function, whose exterior derivative $-\dd x_{d+1}$ is the exterior unit conormal, and 
\begin{align*}
     \nabla^2 x_{d+1} (\xi \otimes \xi) = \frac{\d^2}{\d t^2} \bigg(\cos t\, x_{d+1}+\sin t \, \xi_{d+1}\bigg)_{t=0} = -\alpha_0
\end{align*}
thus the second fundamental form of $\d S^d_{>\alpha_0}$ is definite positive if and only if $\alpha_0>0$. The exponential map $\exp_{(x',0)}: T_{(x',0)}S^d_+ \to  S^d_+ $ is a diffeomorphism from $\{\xi \in T_{(x',0)}S^d_+ : \xi_{d+1}>0\}$ to $S^d_+$:
\begin{align*}
     \text{if } \eta = \exp_{(x',0)}^{-1}(y), \quad y \in S^d_+ &\text{ then }
    |\eta| = d_{S^{d}}((x',0),y)= \arccos \langle x',y'\rangle \\ &\text{ and }
    \frac{\eta}{|\eta|} = \frac{y-\langle x',y'\rangle \, (x',0)}{\sqrt{1- \langle x',y'\rangle^2}}.
\end{align*}
This shows that any cap stricly smaller than the hemisphere is a \textit{simple} manifold.
Of course, the former inverse breaks down when $\langle x',y' \rangle=\pm 1$ i.e. when $y$ and $(x',0)$ are equal or conjugate points (which may occur only on $\overline{S^{d}_+}$), giving further evidence that the full northern hemisphere $\overline{S^{d}_+}$ is far from being a \textit{simple} manifold. 

The aim of this section is to prove that the attenuated geodesic transform restricted to functions supported in the open northern hemisphere is injective for small values of the attenuation parameter $\lambda>0$ and to derive the corresponding stability estimates. 
We will restrict our attention to functions which are supported in a cap strictly smaller than the northern hemisphere
     $$ S^d_{>\alpha_0} = \big\{x \in S^d : x_{d+1} > \alpha_0 \big\}, \quad 0<\alpha_0<1, $$
(a simple manifold). 
We will use the following spaces and norms: on $\d_+S(S^d_+)$ one introduces the measure
     $$ \dd\mu = \xi_{d+1} \, \dd x' \, \dd \xi $$
(where $\dd x'$ is the measure on the hypersphere $S^{d-1}$ and $\dd \xi$ the measure on the hypersphere $S^d \cap (x',0)^{\perp}$) and the corresponding $L^2$ norm
     \begin{align*}
         \|F\|_{L^2(\d_+S(S^d_+))}  &= \bigg(\int_{\d_+S(S^d_+)} |F(x',0,\xi)|^2\, \dd \mu( x',\xi) \bigg)^{\frac{1}{2}} \\ 
         &=  \bigg(\int_{S^{d-1}} \int_{S^d_+ \cap (x',0)^{\perp}}  |F(x',0,\xi)|^2\,  \xi_{d+1} \,\dd \xi \, \dd x' \bigg)^{\frac{1}{2}}.
     \end{align*}
This choice is motivated by Santal\'o's formula%
\footnote{Santal\'o's formula \cite{S} is slightly more general and actually applies to integrals on the cosphere bundle $S(S^d_+)$.}
\begin{align*}
     \int_{S^d_+} f(x) \, \dd x = \frac{1}{|S^{d-1}|} \int_{\d_+S(S^d_+)} \int_0^{\pi} f(\exp_{(x',0)}(t \xi)) \,\dd t \, \dd \mu
\end{align*}
which immediately implies the $L^2$ continuity of the attenuated geodesic ray transform
\begin{align*}
      \int_{\d_+S(S^d_+)} |T^+_{\lambda}f|^2 \, \dd \mu &\leq \pi  \int_{\d_+S(S^d_+)} \int_0^{\pi} |f(\exp_{(x',0)}(t \xi))|^2 \,\dd t \, \dd \mu \\ &= \pi |S^{d-1}| \, \int_{S^d_+} |f(x)|^2 \, \dd x.
\end{align*}
\begin{rem}
\label{Xray:SobolevContinuityRem}
     Note that again from Santalo's formula
     \begin{align*}
      \int_{\d_+S(S^d_+)} |\dd T^+_{\lambda}f|^2 \, \dd \mu &\leq \pi  \int_{\d_+S(S^d_+)} \int_0^{\pi} |\dd f(\exp_{(x',0)}(t \xi))|^2 \,\dd t \, \dd \mu \\ &= \pi |S^{d-1}| \, \int_{S^d_+} |\dd f(x)|^2 \, \dd x.
\end{align*}
     it is easy to get the boundedness of the geodesic ray transform from $H^1(S^{d}_+)$ to $H^1(\d_+S(S^d_+))$. By interpolation, this transform is also continuous from $H^{\sigma}(S^{d}_+)$ to $H^{\sigma}(\d_+S(S^d_+))$ when $\sigma \in[0,1]$.
\end{rem}
The following result is classical \cite{FSU,SU,Sh} and the reader familiar with geometric tomography can skip the section. 
We provide proofs in the case of the hyperpshere because computations are fairly explicit in this case and can serve as an introduction to such tomography problems on manifolds.
\begin{thm}
\label{Xray:AttXraySpherEst}
     Let $\alpha_0 \in (0,1)$, there exist two constants $\lambda_0>0$ and $C>0$ such that for all
     $0 \leq \lambda \leq \lambda_0$ and all $f \in L^2(S^d_+)$ supported in $S^d_{>\alpha_0}$ one has the following estimate
     \begin{align*}
          C\|T_{\lambda}^+f\|_{L^2(\d_+S(S^d_+))} \geq \|f\|_{H^{-\frac{1}{2}}(S^d_+)}.
     \end{align*}
\end{thm}
\begin{rem}
     If such a stability estimate holds for large values of the attenuation parameter $\lambda>0$ then the constant $C$ will behave exponentially badly, since if one takes $f$ to be a smooth function supported in $x_{d+1}>1/2$, normalized in such a way that the $L^2$ norm is one, then by Santalo's formula
     $$ \int_{\d_+S(S^d_+)} |T_{\lambda}^+f(x',\xi)|^2 \, \dd \mu  \leq \pi|S^{d-1}| \, \e^{-\lambda \pi/4} \ $$
     and therefore
     $$ C \geq \frac{\e^{\lambda \pi/4} }{\pi |S^{d-1}| } \, \|f\|_{H^{-1/2}(S^d_+)} . $$
     In particular, from the point of view of stability estimates for the Calder\'on problem with partial data, this means that, were we to prove a stability estimate for the attenuated geodesic transform for a wider range of values of the attenuation parameter $\lambda>0$, we would still be obtaining $\log$-$\log$ type stability estimates for the partial Dirichlet-to-Neumann map (one logartihmic loss for each of the  exponential factors arising from the exponential behaviour of the complex geometrical optics and from the exponential behaviour of the constant in the stability estimate for the attenuated ray transform). 
     For efficiency purposes it is therefore not worth pursuing the investigation of estimates for larger values of the attenuation parameter, and for our goal, we are satisfied with small values of $\lambda$. 
     Note that in \cite{DSFKSU}, the injectivity of the attenuated geodesic ray transform was proved for small values of the attenuation in simple compact Riemannian manifolds with boundary, and that in \cite{SaU}, the injectivity was obtained for all attenuations on simple surfaces.
\end{rem}

\subsection{Stereographic projection}
\label{Par:Stereo}

To prove the estimate, we will study the corresponding Euclidean weighted X-ray transform obtained by mapping the northern hemisphere to the Euclidean space through stereographic projection
\begin{align*}
     \sigma_+ : S^d_+ \to \R^{d}
\end{align*}
the expression of which is 
\begin{align*}
     \sigma_+(x) = \frac{x'}{x_{d+1}}, \quad \sigma_+^{-1}(z)=\frac{(z,1)}{\langle z \rangle}.
\end{align*}
Half great circles are transformed into lines by the stereographic projection
    $$ \sigma_+ \circ \gamma_{x',\xi}(t) =\mathop{\rm cotan} t \, \frac{x'}{\xi_{d+1}} + \sigma_+(\xi). $$ 
In the coordinates given by the projection, the ray transform reads 
    $$ T_{\lambda}^+f(x',\xi) = T^{\infty}_{\lambda}(f\circ\sigma_+^{-1})(\sigma_+(\xi),x'/\xi_{d+1}) $$
where $T^{\infty}_{\lambda}$ is the following weighted (Euclidean) X-ray transform
\begin{align*}
    T^{\infty}_{\lambda}g(z,\zeta) = \int_{-\infty}^{\infty} g(z+t\zeta) \, \e^{-\lambda \big(\frac{\pi}{2}-\arctan t \big)}\, \frac{\dd t}{1+t^2}, \\ z \in \R^d, \zeta \in z^{\perp}, |\zeta|= \langle z \rangle.
\end{align*}        
Note that if $f$ is supported in $S^d_{>\alpha_0}$ then its pullback by the inverse of the stereographic projection $g=f \circ \sigma^{-1}_+$ is supported in the ball $B(0,(\alpha_0^{-2}-1)^{1/2})$ where $\langle z \rangle \leq \alpha_0^{-1}$.
\begin{rem}
     The stereographic projection is an isometry between the hemisphere $(S^{d}_+,g_{d})$ endowed with the canonical metric and $(\R^d,G_d)$ where the metric $G_d$ is given (in matrix form) by
      $$ G_d = \frac{1}{1+|z|^2}\bigg({\rm I}_d- \frac{ z \otimes z}{1+|z|^2}\bigg). $$
In particular, we have
      $$ \sqrt{\det G_d} = \langle z \rangle^{-d-1} $$
and the corresponding $L^2$ space obtained by projecting square integrable functions on the hemisphere by stereographic projection is     
     $$ L^2(\R^d,\langle z \rangle^{-d-1}\dd z). $$
In terms of norms, this gives 
     $$ \|f\|^2_{L^2(S^d_+)} = \int_{\R^d} |f \circ \sigma_+^{-1}(z)|^2 \langle z \rangle^{-d-1} \, \dd z. $$
\end{rem}
Our aim is to control the $H^{-1/2}$ norm of $f$ by
\begin{align*}
     \|T_{\lambda}^+f\|_{L^2(\d_+S(S^d_+))}^2 &= \int_{\d_+S(S^d_+)} |T_{\lambda}^+f(x',\xi) |^2 \, \dd \mu(x',\xi)
    \\ &= \int_{S^{d-1}} \bigg(\int_{\zeta^{\perp}} \big|T^{\infty}_{\lambda}(f\circ \sigma_+^{-1})(z,\langle z \rangle \zeta)\big|^2 \langle z \rangle^{-1-d} \, \dd z\bigg) \, \dd \zeta
\end{align*}
(where $\mu$ is the measure on $\d_+S(S^d_+)$ defined in the beginning of the section).
The weighted ray transform under scope also reads    
 \begin{align*}
    T^{\infty}_{\lambda}g(z,\langle z \rangle\zeta) &= \langle z \rangle  \int_{-\infty}^{\infty} g(z+t\zeta) \, \e^{-\lambda \big(\frac{\pi}{2}-\arctan \frac{t}{\langle z \rangle}\big)}\, \frac{\dd t}{t^2+\langle z \rangle^2} \\ &= \langle z \rangle X_{w_{\lambda}} \big(\langle z \rangle^{-2} g\big)(z,\zeta), \qquad z \in \R^d, \zeta \in z^{\perp}, |\zeta| = 1,
\end{align*}        
where  $X_w$ denotes the Euclidean weighted X-ray transform   
    $$ X_{w}f(z,\zeta) = \int_{-\infty}^{\infty} f(z+t\zeta) \, w(t,z) \, \dd t, \qquad z \in \R^d, \zeta \in S^{d-1} \cap z^{\perp} $$
with (bounded) exponential weight
\begin{align*}
     w_{\lambda}(t,z) = \e^{-\lambda\big(\frac{\pi}{2}-\arctan \frac{t}{\langle z \rangle}\big)}.
\end{align*}    
If the function $g$ is supported in the set $\langle z \rangle \leq \alpha_0^{-1}$ then so is the weighted ray transform $X_wg$ and thus we have 
\begin{multline*}
     \bigg(\int_{S^{d-1}} \int_{\zeta^{\perp}} \big| \langle z \rangle^{-\frac{d-1}{2}} X_{w_{\lambda}}(\langle z \rangle^2 g)\big|^2 \, \dd z \, \dd \zeta\bigg) \\ \geq \alpha_0^{\frac{d-1}{2}} \bigg(\int_{S^{d-1}} \int_{\zeta^{\perp}} \big|X_{w_{\lambda}}(\langle z \rangle^2 g)\big|^2 \, \dd z \, \dd \zeta\bigg). 
\end{multline*}
Similarly the fact that $\langle z \rangle^{-(d+1)/2} \langle D_z \rangle^{-1/2} \langle z \rangle^{-2}$ is a pseudodifferential operator of order $-1/2$ implies 
\begin{align*}
     \int \big|\langle D_z \rangle^{-\frac{1}{2}} g\big|^2 \langle z \rangle^{-d-1} \, \dd z \leq 
     C' \|\langle z \rangle^2 g\|_{H^{-\frac{1}{2}}(\R^d)} 
\end{align*}
and therefore the stability estimate we propose to prove is equivalent to 
\begin{align}
\label{Xray:KeyWEstimate}
     C'' \bigg(\int_{S^{d-1}} \int_{\zeta^{\perp}} |X_{w_{\lambda}} g(z,\zeta)|^2 \, \dd z \, \dd \zeta\bigg)^{\frac{1}{2}} \geq \|g\|_{H^{-\frac{1}{2}}(\R^d)}
\end{align}
for functions $g \in L^2(\R^d)$ supported in $B(0,(\alpha_0^{-2}-1)^{1/2})$. 

\subsection{The Euclidean weighted X-ray transform}

Motivated by the former computations related to the stereographic projection, we study the Euclidean weighted X-ray transform   
    $$ X_{w}f(z,\zeta) = \int_{-\infty}^{\infty} f(z+t\zeta) \, w(t,z) \, \dd t, \qquad z \in \R^d, \zeta \in S^{d-1} \cap z^{\perp} $$
where $w$ is a smooth function of $(t,z) \in \R^{d+1}$. For further use, we denote $E$ the sphere bundle%
\footnote{From this point of view, we should write $(\zeta,z) \in E$ for the variables but we prefer to abuse notations and keep the former notations. }
    $$ E = \bigcup_{\zeta \in S^{d-1}} \zeta^{\perp}. $$  
\begin{rem}
\label{Xray:XwL2Bound}
     Note that $X_w$ is a bounded operator from $L^2(\R^d)$ to $L^2(E)$ provided $w$ has a good behaviour (for instance if it is a smooth compactly supported function).
     Indeed, by Cauchy-Schwarz in the time integral, we get
     \begin{align*}
           \|X_w f\|_{L^2(E)} &\leq \bigg(\int_{S^{d-1}}  \bigg(\int_{\zeta^{\perp}} \int_{-\infty}^{\infty} \|w(\cdot,z)\|^2_{L^2(\R)} |f(z+t\zeta)|^2 \, \dd t \, \dd z\bigg) \, \dd \zeta\bigg)^{\frac{1}{2}} \\
           &\leq \sup_{z \in \R^d} \|w(\cdot,z)\|_{L^2(\R)} |S^{d-1}|^{\frac{1}{2}}  \|f\|_{L^2(\R^d)}
     \end{align*}
\end{rem}
\begin{thm}
\label{Xray:XwStabEst}
       Let $\alpha_0 \in (0,1)$, there exist two constants $\eps >0$ and $C>0$ and an integer $N$ such that for all weights $w \in \mathcal{C}^{\infty}(\R^{d+1})$ such that
             $$ \sup_{t^2+\langle z \rangle^2 \leq 4\alpha_0^{-2}} |\d^{\beta}(w(t,z)-1)| \leq \eps, \quad |\beta| \leq N$$ 
       and all $f \in L^2(\R^d)$ supported in $B(0,(\alpha_0^{-2}-1)^{1/2})$ one has the following estimate
     \begin{align*}
          C\|X_w f\|_{L^2(E)} \geq \|f\|_{H^{-\frac{1}{2}}(\R^d)}.
     \end{align*}
\end{thm}
First note that since one is interested in lower bounds on $X_wf$ and since norms of
\begin{align*}
     X_wf(z,\zeta)+X_wf(z,-\zeta) = Xf_{\tilde{w}}(z,\zeta), \quad \tilde{w}(t,\zeta)=w(t,\zeta)+w(-t,\zeta)
\end{align*}
are controlled by twice the norms of $X_wf$, we may assume without loss of generality that $w$ is even in time. Besides if $\chi \in \D(0,2\alpha_0^{-1})$ then
   $$ X_w\big(\chi(\langle z \rangle) f\big) = X_{\chi(\langle t,z \rangle) w} f $$
therefore if $f$ is supported in $\langle z \rangle \leq \alpha_0^{-1}$, by choosing $\chi \in \D(0,2\alpha_0^{-1})$ equal $1$ on $[0,\alpha_0^{-1}]$, we may assume that $w$ is a smooth compactly supported function and that
\begin{align}
\label{Xray:SmallControlw}
  \sup_{\R^{d+1}} |\d^{\beta}(w(t,z)-1)| \leq C \eps, \quad |\beta| \leq N. 
\end{align}
  
It is well known that the normal operator built from the Euclidean X-ray transform is a multiple of the square root of the inverse of the standard Euclidean Laplacian
     $$ X_0^*X_0 =  c_d |D_z|^{-1},\quad c_d = 4 (2\pi)^{\frac{d-1}{2}} \Gamma\bigg(\frac{d-1}{2}\bigg) $$
where the adjoint $X_0^*$ of the X-ray transform is given by
     $$ X_0^*F(z) = \int_{S^{d-1}} F(z-\langle z,\zeta\rangle \zeta,\zeta) \, \dd \zeta.  $$ 
This implies that $X_0$ is a bounded isomorphism $\dot{H}^{-1/2}(\R^d) \to L^2(E)$ with inverse
    $$ X_0^{-1} = c_d^{-1} |D_z| X_0^* $$
(where $\dot{H}^{-1/2}$ denotes the homogeneous Sobolev space of order $-1/2$). 
Let us now perform the computation in the weighted case.  The adjoint operator reads
\begin{align*}
     (X_w^*F)(z) = \int_{S^{d-1}} w(\langle z,\zeta\rangle,\pi_{\zeta^{\perp}}(z)) F(\pi_{\zeta^{\perp}}(z),\zeta)  \, \dd \zeta
\end{align*}
where $\pi_{\zeta^{\perp}}(z)=z-\langle z,\zeta\rangle \zeta$ is the orthogonal projection of $z$ on the orthogonal of $\zeta$. The normal operator $N_w=X_w^*X_w$ is therefore given by
\begin{align*}
     N_wf(z) &= 2 \int_{S^{d-1}} \int_{0}^{\infty} w(\langle z,\zeta\rangle,\pi_{\zeta^{\perp}}(z)) \, w(t+\langle z,\zeta \rangle,\pi_{\zeta^{\perp}}(z)) f(z+t\zeta ) \, \dd t \, \dd \zeta \\ &= \int N_w(z,y) f(y) \, \dd y 
\end{align*}     
where by a slight abuse of notations, we have kept $N_w$ to denote the kernel of the normal operator 
\begin{align*}
     N_w(z,y) &= |z-y|^{-d+1}  L_w(z,y)L_w(y,z) \\ L_w(z,y) &= w \bigg(\frac{\langle z,y-z \rangle}{|y-z|},z-\frac{\langle z,y-z\rangle}{|y-z|^2}(y-z)\bigg).
\end{align*}
Note that this kernel is singular on the diagonal $z=y$ on $\R^d_z \times \R^d_y$ and that for $w=0$, we get that $N_0=X_0^*X_0=c_d |D_z|^{-1}$ as explained earlier on.
\begin{lem}
\label{Xray:Pseudo}
      Let $w \in \D(\R^{n+1})$,  let $\psi \in \D(\R^d)$ equal one on the unit ball, the  component of the normal operator acting on high frequencies
         $$ N^{\rm high}_w = N_w(1-\psi(D_z)) $$ 
      is a pseudodifferential operator on $\R^d$ of order $-1$, the symbol of which has seminorms in $S^{-1}$ bounded by finitely many seminorms of $w$ in $\D(\R^{n+1})$ of the form
          $$ \sup_{\R^{d+1}} |\d^{\beta}w|. $$  
\end{lem}
\begin{proof}
     The symbol associated with the normal operator is
     \begin{align*}
           a_{w}(z,\zeta) &= (1-\psi(\zeta)) \int N_w (z,z-y)\e^{- i y \cdot \zeta} \, \dd y.
     \end{align*}
      We will use a dyadic partition of unity
           $$ 1 = \sum_{\mu=-\infty}^{\infty} \chi(2^{-\mu}y), $$
with $\chi$ a function supported in the annulus $\frac{1}{2} \leq |y| \leq 2$, to decompose the symbol as a sum of terms of the form
\begin{multline*}
      2^{\mu d} (1-\psi(\zeta))\int \e^{i 2^{\mu} y\cdot \zeta}  \chi(y) N_{w}(z,z-2^{\mu}y) \, \dd y \\ 
      = 2^{\mu}  (1-\psi(\zeta)) \int \e^{i 2^{\mu} y\cdot \zeta}  |y|^{-d+1}\chi(y) L_{w}(z,z-y)L_w(z-2^{\mu}y,z)  \, \dd y 
\end{multline*}
with
\begin{align*}
       L_{w}(z,z-y) &= w\bigg(-\frac{\langle z,y \rangle}{|y|},\pi_{y^{\perp}}(z)\bigg), \\
       L_w(z-2^{\mu}y,z) &=w\bigg(-\frac{\langle z,y \rangle}{|y|}-2^{\mu}|y|,\pi_{y^{\perp}}(z)\bigg).
\end{align*}
Note that because of the compact support of the weight $w$, the first function $L_w(z,z-y)$ has compact support in $z$. Derivatives with respect to $y$ of the amplitude
     $$ |y|^{-d+1}\chi(y) L_{w}(z,z-y)L_w(z-2^{\mu}y,z) $$
are bounded by finitely many seminorms of $w$ in $\D(\R^{n+1})$. Applying the non-stationary phase when $2^{\mu}|\zeta|$ is larger than $1$ we get 
     \begin{align*}
          |a_w(z,\zeta)| &\leq C(w)  \sum_{2^{\mu}|\zeta| \geq 1} 2^{\mu} (2^{\mu}|\zeta|)^{-M}
          + C(w)  \sum_{\, 2^{\mu}|\zeta| \leq 1}  2^{\mu} \\ &\leq C(w)\sqrt{2} \bigg(\frac{2^M}{M-1}+2\bigg)\langle \zeta \rangle^{-1}
     \end{align*}
with a constant $C(w)$ depending on finitely many seminorms of $w$.
Repeating this argument for the derivatives of this function, we get that $a_{w}$ is a classical symbol of order $-1$.
\end{proof}    
\begin{clry}
\label{Xray:XwSobBound}
     The weighted X-ray transform $X_w$ is a bounded operator from $H^{-\frac{1}{2}}(\R^d)$ to $L^2(E)$. Its norm is bounded by finitely many seminorms of $w$ in $\D(\R^{n+1})$.
\end{clry}
\begin{proof}
     By Lemma \ref{Xray:Pseudo}, $N^{\rm high}_w$ is a bounded operator from $H^{-\frac{1}{2}}(\R^d)$ to $H^{\frac{1}{2}}(\R^d)$ and this fact together with Remark \ref{Xray:XwL2Bound} implies the following bound
     \begin{align*}
          \|X_w f\|_{L^2(E)}^2 &= \|X_w \psi(D_z) f\|_{L^2(E)}^2 + \langle N_w^{\rm high} f,(1-\psi(D_z))f\rangle_{L^2(\R^d)}  \\ &\leq C^2\|\psi(D_z)f\|^2_{L^2(\R^d)} + C^2 \|f\|^2_{H^{-\frac{1}{2}}(\R^d)}. 
     \end{align*}
     This yields the corollary since $\psi(D_z)$ is a bounded operator from $H^{-1/2}$ to $L^2$ and the constant $C$ depends on  finitely many seminorms of $w$ in $\D(\R^{n+1})$.
\end{proof}

\subsection{A perturbation argument}

To prove Theorem \ref{Xray:XwStabEst}, we will use a perturbation argument and compare the attenuated ray transform with the classical X-ray transform: 
\begin{align*}
    \|f\|_{H^{-\frac{1}{2}}(\R^d)}  &\leq C_1  \|X_0f\|_{L^2(E)} \\
    &\leq C_1\|X_w f\|_{L^2(E))} + C_1\|X_{\chi(\langle t,z\rangle)(w-1)} f\|_{L^2(E)}
\end{align*}
Thanks to Corollary \ref{Xray:XwSobBound} and \eqref{Xray:SmallControlw}
\begin{align*}
    \|f\|_{H^{-\frac{1}{2}}(\R^d)}  
    &\leq C_1\|X_w f\|_{L^2(E))} + C_2 \eps \|f\|_{H^{-\frac{1}{2}}(\R^d)}
\end{align*}
and taking $\eps=1/2C_2$ is enough to prove the estimate in Theorem \ref{Xray:XwStabEst}.
This implies estimate \eqref{Xray:KeyWEstimate} when $\lambda$ is small since 
      $$  \sup_{t^2+\langle z \rangle^2 \leq 4\alpha_0^2} |\d^{\beta} (w_{\lambda}(t,z)-1)| =  \sup_{t^2+\langle z \rangle^2 \leq 4\alpha_0^2} \Big|\d^{\beta} \Big(\e^{-\lambda\big(\frac{\pi}{2}-\arctan \frac{t}{\langle z \rangle^2}\big)}-1\Big)\Big| $$
is small provided $\lambda$ is, and thus also completes the proof of Theorem \ref{Xray:AttXraySpherEst} after use of the stereographic projection.
\begin{rem}
     This (standard) perturbation argument allows to deduce stability estimates (and in particular injectivity) for weighted geodesic ray transforms on simple compact Riemannian manifolds with boundary from unweighted ones provided the weight is close enough to 1.
\end{rem}

%
%
\section{Complex geometrical optics}
\label{sec:CGO}

In this section, we construct the solutions of the Schr\"odinger equation by complex geometrical optics with logarithmic weights introduced by Kenig, Sj\"ostrand and Uhlmann. The presentation follows further developments for the anisotropic case \cite{DSFKSU,DSFKLS,KS} but restricted to the Euclidean case. Readers interested in a more conceptual and general presentation (on manifolds) are {referred} to \cite{DSFKLS,KS}.

\subsection{The Euclidean space as a warped product}\label{ssec:ccgo1}

Let $\Omega \subset \R^n$ be a bounded open set with smooth boundary, let $x_0 \notin \ch(\overline{\Omega})$, there exists $\eps>0$ such that
$B(x_0,\eps) \cap \ch(\overline{\Omega}) = \varnothing$. One can separate the convex hull  $\ch(\overline{\Omega})$ of the ball $B(x_0,\eps)$ by a hyperplane $H_{\omega_0,s_0}$ of equation $\langle x-x_0,\omega_0 \rangle = s_0$ with $\omega_0 \in S^{n-1}$ and $s_0 \geq \eps >0$
\begin{align*}
     \ch(\overline{\Omega}) &\subset H^+_{\omega_0,s_0} = \big\{x \in \R^n : \langle x-x_0,\omega_0 \rangle > s_0 \big\} \\
     B(x_0,\eps) &\subset  H^-_{\omega_0,s_0} = \big\{x \in \R^n : \langle x-x_0,\omega_0 \rangle < s_0 \big\}.
\end{align*}
Besides, there exists $\varrho_0 > \eps >0$ such that $\ch(\overline{\Omega}) \subset B(x_0,\varrho_0)$. Let $\alpha_0 = s_0/\varrho_0 \in (0,1]$, we consider the cap and hemisphere of the hypersphere with north pole 
$\omega_0$
\begin{align*}
     S^{n-1}_{>\alpha_0} &= \big\{ \omega \in S^{n-1} : \langle \omega,\omega_0 \rangle > \alpha_0 \big\} \subset S^{n-1}_+ \\
     S^{n-1}_+ &= \big\{ \omega \in S^{n-1} : \langle \omega,\omega_0 \rangle > 0 \big\}.
\end{align*}
Part of the subsequent construction is based on the observation that
\begin{align*}
     \Gamma = \bigg\{ \frac{x-x_0}{|x-x_0|} : x \in \overline{\Omega} \bigg\} \subset S^{n-1}_{>\alpha_0}.
\end{align*}
One considers the following coordinates in $\overline{\Omega}$
\begin{align*}
     t = \log|x-x_0| \in \R, \quad \omega = \frac{x-x_0}{|x-x_0|} \in \Gamma  \subset S^{n-1}_{>\alpha_0},
\end{align*}
with respect to which the Laplace operator reads
\begin{align*}
     \Delta = \e^{-2t} (\d_t^2 +(n-2)\d_t + \Delta_{S^{n-1}}).
\end{align*}
By conjugation, this expression can be further reduced into
\begin{align*}
     \Delta = \e^{-\frac{n+2}{2}t} \big(\d_t^2 + \widehat{\Delta}_{S^{n-1}}\big)\e^{\frac{n-2}{2}t}
\end{align*}
where $\widehat{\Delta}_{S^{n-1}}=\Delta_{S^{n-1}}-(n-2)^2/4$.
Note that in those coordinates, the Lebesgue measure becomes 
     $$ \dd x = \e^{nt} \, \dd t \wedge \dd \omega, $$
{where $\dd \omega$ denotes the Lebesgue measure on $S^{n - 1}$}.
\begin{rem}
     The former computations are based on the fact that using polar coordinates $(r,\omega)$ centered at $x_0$, $\Omega$ can be imbedded in $\R_+ \times S^{n-1}_+$, and in polar coordinates the Euclidean metric looks like
       $$ |\dd x|^2 = \dd r^2 + r^{2} g_{{n-1}} $$
where $g_{n-1}$ is the canonical metric on the hypersphere $S^{n-1}$.   
This means that $\Omega$ can be seen as a submanifold of the warped product $(\R_+,\dd r^2) \times_{r^{-2}} (S^{n-1},g_{{n-1}})$ of the Euclidean line with the hypersphere. 

Warped profucts with an Euclidean factor are manifolds for which there exist limiting Carleman weights as shown in \cite{DSFKSU,DSFKS}. Existence of limiting Carleman weights is a necessary condition to perform the construction of solutions of the Schr\"odinger equation by use of complex geometric optics in the spirit%
\footnote{That is with antagonizing exponential behaviour.}
of Sylvester and Uhlmann~\cite{SU}. One can unwarp the metric by making the change of variable $t=\log r$ as observed in~\cite[Remark {1.4}]{DSFKS}.    
\end{rem}

We are now interested in approximate solutions of the Schr\"odinger equation of the form
    $$ u = \e^{-\big(s+\frac{n-2}{2}\big)t} v_s(\omega), \quad s \in \H  $$
where $v_s$ is a quasimode of $\widehat{\Delta}_{S^{n-1}}$  
\begin{align*}
     \big(\widehat{\Delta}_{S^{n-1}}+s^2\big)v_s = \mathcal{O}_{L^2(S^{n-1}_+)}(1).
\end{align*}
With this choice, we have
\begin{align*}
     \Delta u &=  \e^{-\big(s+\frac{n+2}{2}\big)t} \bigg(\Delta_{S^{n-1}}-\frac{(n-2)^2}{4}+s^2\bigg)v_s
\end{align*}
and therefore $u$ is an approximate solution%
\footnote{This is justified by the fact that the behaviour with respect to $s$ is \textit{better} than the \textit{expected} one, which is of $|s|^2 \e^{- \re s \, t}$ because of the fact that two derivatives fall on $\e^{st}$.}
of the Schr\"odinger equation
\begin{align*}
     \| {|x - x_0|^s} (-\Delta+q)u\|_{L^2{(\Omega)}} \leq C.
\end{align*}

\subsection{Quasimodes on the hemisphere}\label{ssec:ccgo2}
Recall that we are considering the hemisphere $S^{n-1}_+$ with north pole $\omega_0$. We will look for quasimodes of the form
\begin{align*}
     v_s(\omega) = \e^{i s d_{S^{n-1}}(\omega,y)} a, \quad y \in \d S^{n-1}_+.
\end{align*}
The construction may be done as in \cite{DSFKSU} by using the fact that the phase $\psi(\omega)=d_{S^{n-1}}(\omega,y)$ solves the eikonal equation
     $$ |\dd \psi|^2_{g_{n-1}} = 1 $$
and by taking $a$ to be the solution of a transport equation. Here, we choose to be slightly more explicit. We refer to \cite{DSFKSU,DSFKLS,KS} for a general approach on manifolds.

{For $y \in \partial S^{n - 1}_+$,} we parametrize the unit sphere $S^{n-1}$ in the following way
\begin{align*}
    \omega = (\cos \theta) y + (\sin \theta) \eta, \quad  
    \theta \in [0,\pi], \quad \eta \in S^{n-1} \cap y^{\perp} \simeq S^{n-2}
\end{align*}
so that $\theta = d_{S^{n-1}}(\omega,y)$ and $\omega=\exp_{y} (\theta \eta)$. The canonical Riemannian metric on the sphere reads in those coordinates
\begin{align*}
      g_{n-1} = \dd \theta^2  + (\sin^2 \theta) \, g_{n-2} \quad \theta \in [0,\pi] 
\end{align*} 
and in particular
    $$ \sqrt{\det g_{n-1}} = (\sin \theta)^{n-2} \sqrt{\det g_{n-2}}. $$
Note that again this canonical metric has the form of a warped product.
In those coordinates, the Laplace-Beltrami operator on the sphere reads%
\footnote{In particular when $n=3$, one recovers the well-known expression of the Laplacian-Beltrami operator on the two dimensionnal sphere
\begin{align*}
     \Delta_{S^{2}} = \frac{\d^2}{\d\theta^2} + \cot \theta \, \frac{\d}{\d\theta} + \frac{1}{\sin^2 \theta} \frac{\d^2}{\d\eta^2}.
\end{align*}}
\begin{align*}
     \Delta_{S^{n-1}} = \frac{1}{(\sin \theta)^{n-2}}\frac{\d}{\d\theta} \bigg((\sin \theta)^{n-2} \frac{\d}{\d\theta}\bigg) + \frac{1}{\sin^2 \theta} \, \Delta_{S^{n-2}}.
\end{align*}
As before, the expression of the Laplacian operator may be further simplified by conjugation: 
\begin{align*}
     \Delta_{S^{n-1}} = \frac{1}{(\sin \theta)^{\frac{n-2}{2}}}\frac{\d^2}{\d\theta^2} (\sin \theta)^{\frac{n-2}{2}}  + \frac{1}{\sin^2 \theta} \, \Delta_{S^{n-2}} +  \frac{1}{(\sin \theta)^{n-2}} [L^{\dag},L]
\end{align*}
where $L=(\sin \theta)^{\frac{n-2}{2}} \d/\d\theta$ and $L^{\dag} ={ \partial/\partial\theta (\sin \theta)^\frac{n-2}{2}}$.

Computing the commutator 
\begin{align*}
     [L^{\dag},L]&=-(\sin \theta)^{\frac{n-2}{2}} \frac{\d^2}{\d\theta^2}\Big((\sin \theta)^{\frac{n-2}{2}}\Big) \\ &= \frac{(n-2)^2}{4} (\sin \theta)^{n-2} - \frac{(n-2)(n-4)}{4} (\sin \theta)^{n-4}        
\end{align*}
gives  the following expression for the shifted Laplace-Beltrami operator $\widehat{\Delta}_{S^{n-1}}$ on the hypersphere
\begin{align}
\label{0:ConjugatedLaplaceSphere}
     \widehat{\Delta}_{S^{n-1}} =   (\sin \theta)^{-\frac{n-2}{2}} \bigg(\d_{\theta}^2 + \frac{1}{\sin^2 \theta} \, \widetilde{\Delta}_{S^{n-2}}\bigg)  (\sin \theta)^{\frac{n-2}{2}}
\end{align}
with 
    $$  \widetilde{\Delta}_{S^{n-2}} = \Delta_{S^{n-2}} - \frac{(n-2)(n-4)}{4}. $$
We note that in those coordinates, we have
\begin{align*}
     \omega \in S^{n-1}_{+} \Rightarrow \langle \omega,\omega_0 \rangle = \sin \theta \langle \eta,\omega_0 \rangle > 0
\end{align*}
hence $\theta \in (0,\pi)$ and $\eta \in S^{n-1}_+ \cap y^{\perp}$.

Now, we can look for quasimodes using separation of variables
\begin{align*}
      v_s(\omega) = (\sin \theta)^{-\frac{n-2}{2}} \alpha(\theta) b(\eta), 
\end{align*}
where $\alpha$ and $b$ are respective quasimodes for the operators $\d_{\theta}^2+s^2$ and $\tilde{\Delta}_{S^{n-2}}$
 \begin{align*}
       (\d_{\theta}^2+s^2)\alpha = \mathcal{O}(1), \quad 
       \widetilde{\Delta}_{S^{n-2}}b= \mathcal{O}(1) \quad \text{ as }
       |s| \to \infty.
\end{align*}
This leads for instance to the choice $\alpha=\e^{is \theta}$ and $b$ to be any smooth function of $\eta$ independent of $s$. We will therefore work with the following quasimode
\begin{align*}
     v_s(\omega) = (\sin \theta)^{-\frac{n-2}{2}} \e^{i s \theta} b(\eta), \quad \omega = \exp_y(\theta \eta)
\end{align*}
for which we have indeed
\begin{align*}
      \int_{S^{n-1}_+} \big|\big(\widehat{\Delta}_{S^{n-1}}+s^2\big)
      v_s(\omega)\big|^2 \, \dd \omega &= \int_{S^{n-2}_+} \int_0^{\pi} 
      \e^{-(\im s) \theta} \big|\widetilde{\Delta}_{S^{n-2}} b\big|^2\, \dd \theta \, \dd \eta 
 \\ &= \bigg(\frac{1-\e^{-(\im s) \pi}}{\im s}\bigg) \, \|\widetilde{\Delta}_{S^{n-2}} b\|^2_{L^2(S_y(S^{n-1}_+))} \\ &\leq \pi \|\widetilde{\Delta}_{S^{n-2}} b\|^{2}_{L^2(S_y(S^{n-1}_+))},
\end{align*}
when $\im s>0$. {Here $S_y(S^{n - 1}_+) = S^{n - 1}_+ \cap y^\perp$.} The $L^2$ norm of this quasimode may be bounded by
\begin{align*}
     \int_{S^{n-1}_+} |v_s(\omega) |^2 \, \dd \omega &= \int_{S^{n-2}_+} \int_0^{\pi} |v_s|^2 (\sin \theta)^{n-2}\, \dd \theta \, \dd \eta 
 \\ &= \bigg(\frac{1-\e^{-(\im s) \pi}}{\im s}\bigg) \, \|b\|^2_{L^2(S_y(S^{n-1}_+))} \leq \pi \|b\|^{2}_{L^2(S_y(S^{n-1}_+))}.
\end{align*}
Similarly, we can bound the approximate solution $u_s=|x-x_0|^{-\big(s + \frac{n - 2}{2} \big)} v_s$ in some weighted $L^2$ space
\begin{align*}
     \int_{\Omega} |x-x_0|^{2\re s} \, |u_s(x)|^2 \dd x &\leq  \int_{\eps}^{{\varrho_0}}\int_{S^{n-1}_+} \e^{2t} \, |v_s(\omega) |^2 \, \dd t \, \dd \omega \\ &\leq \bigg(\frac{\e^{2{\varrho_0}}-\e^{2\eps}}{2}\bigg) \|b\|^{2}_{L^2(S_y(S^{n-1}_+))}. 
\end{align*}
This completes our construction.

\subsection{Global Carleman estimates}

The estimation of the Dirichlet-to-Neumann map outside of the front set is based on a global Carleman estimate (i.e. a Carleman estimate with boundary terms).
\begin{thm}
\label{CGO:CarlEstThm}
      Let $\Omega \subset \R^n$ be an open bounded set with smooth boundary, let $M>0$ and $x_0 \notin \ch(\overline{\Omega})$, there exist constants
$C>0,\tau_0$ such that for all $\tau \geq \tau_0$, all $q \in L^{\infty}(\Omega)$ such that $\|q\|_{L^{\infty}} \leq M$ and all $u \in C^{\infty}(\Omega) \cap H^1_0(\Omega)$  the following Carleman estimate
     \begin{multline*}
           \| |x-x_0|^{-\tau} \d_{\nu}u\|_{L_w^2(B(x_0))} + \tau^{\frac{1}{2}} \||x-x_0|^{-\tau} u\|_{L^2(\Omega)} + \tau^{-\frac{1}{2}}\||x-x_0|^{-\tau} \nabla u\|_{L^2(\Omega)} 
          \\\leq   C \| |x-x_0|^{-\tau} \d_{\nu}u\|_{L^2_w(F(x_0))} + C\tau^{-\frac{1}{2}}\||x-x_0|^{-\tau} (-\Delta+q)u\|_{L^2(\Omega)}
     \end{multline*}
holds true. The norms of the boundary terms are weighted $L^2$ norms
     $$ L^2_w(\Gamma) = L^2(\Gamma,w\,\dd S), \quad  w(x)=\frac{|\langle \nu, x-x_0\rangle|}{|x-x_0|^2}. $$
\end{thm}
We refer to \cite{KSU} for a proof, to \cite{DSFKSjU} for the case of a magnetic Schr\"odinger equation and to \cite{KS} for an improved version of this estimate. An anisotropic version for Laplace-Beltrami operators on a manifold was given in \cite{DSFKSU}. All of these articles establish estimates for general limiting Carleman weights (see \cite{KSU,DSFKSU} for a definition of this notion), i.e. weights which are degenerate in the context of Carleman estimates. Note that it is easy to see that the constant $C$ is uniform with respect to the potential, provided the potential remains in a ball in $L^{\infty}(\Omega)$ since the estimate is proved first for $q=0$ and then perturbed into the one with potential.

For $\delta>0$, we introduce the set
    $$ F_{\delta}(x_0) = \big\{x \in \d\Omega : \langle x-x_0,\nu \rangle \leq \delta |x-x_0|^2 \big\} \supset F(x_0). $$
We are now in a position to write estimates on the Dirichlet-to-Neumann map.
\begin{clry}
\label{CGO:CarlDNClry}
      Let $\Omega \subset \R^n$ be an open bounded set with smooth boundary, let $M,\delta>0$ and $x_0 \notin \ch(\overline{\Omega})$, there exist constants
$C>0,\tau_0$ such that for all $\tau \geq \tau_0$ all $q_1,q_2 \in L^{\infty}(\Omega)$ such that $\|q_1\|_{L^{\infty}},\|q_2\|_{L^{\infty}} \leq M$ and all $f \in H^{\frac{1}{2}}(\d\Omega)$ we have 
     \begin{multline*}
           \| |x-x_0|^{-\tau} (\Lambda_{q_1}-\Lambda_{q_2})f\|_{L^2(\d \Omega \setminus F_{\delta}(x_0))} 
          \leq   \\ C\||x-x_0|^{-\tau}(\Lambda_{q_1}-\Lambda_{q_2})f\|_{L^2(F(x_0))} + C\tau^{-\frac{1}{2}}\||x-x_0|^{-\tau} (q_1-q_2)u\|_{L^2(\Omega)}
     \end{multline*}
     where $u$ is the solution of the Schr\"odinger equation \eqref{Intro:DirichletSchrod} with potential $q_1$ or $q_2$ and Dirichlet datum $f$.
\end{clry}
\begin{proof}
     By standard regularization procedures, the Carleman estimate of Theorem \ref{CGO:CarlEstThm} remains true for functions $u \in H^1_0(\Omega)$ such that $(-\Delta+q)u \in L^2(\Omega)$. Let $u_1,u_2$ be the solution of the Schr\"odinger equation \eqref{Intro:DirichletSchrod} with respective potential $q_1$ or $q_2$ and Dirichlet datum $f$. We have $u_1-u_2 \in H^1_0(\Omega)$ and 
     $$ (\Delta-q_1)(u_1-u_2) = (q_1-q_2) u_2 \in L^2(\Omega) $$
The estimate follows by applying Theorem {\ref{CGO:CarlEstThm}} to the function $u_1-u_2$
and observing that
\begin{align*}
     \delta \leq w(x),  \; x \in \d\Omega \setminus F_{\delta}(x_0) \subset B(x_0), \quad \text{ and } \quad w(x) \leq |x-x_0|^{-1} 
\end{align*}
to get rid of the weights.
\end{proof}

\subsection{Complex geometrical optics}

We summarize the construction of solutions to the Schr\"odinger equation by use of complex geometrical optics in the following theorems.
\begin{thm}
     Let $M > 0$, $q \in L^{\infty}(\Omega)$ with $\| q \|_{L^\infty} \leq M$ and $x_0 \notin \ch(\bar{\Omega})$, there exists a solution of the Schr\"odinger equation $-\Delta u + qu=0$ of the form
     \begin{align*}
          u_{s}{(x)} &= |x-x_0|^{-s}({|x -x_0|^{-\frac{n - 2}{2}}}v_s(x)+r_{s,q}(x)) 
     \end{align*}
     where $v_s$ is a quasimode of $\hat{\Delta}_{S^{n-1}}$, of the form 
           $$ v_s(x) = \bigg(\sin d_{S^{n-1}}\bigg(\frac{x-x_0}{|x-x_0|},y\bigg)\bigg)^{-\frac{n-2}{2}} \e^{i s d_{S^{n-1}}\big(\frac{x-x_0}{|x-x_0|},y\big)} b $$
     with $y \in \d_+S^{n-1}$ and $b$ any smooth function of the variable $\eta$ if $(x-x_0)/|x-x_0|= \exp_y(\theta \eta)$  and where the remainder $r_{s,q}$ satisfies the estimate
	{\begin{align*}
          \|r_{s,q}\|_{L^2(\Omega)} &\leq C |\re s|^{-1} \|b\|_{H^2(S_y(S^{n-1}_+))}, \\
          \|\nabla r_{s,q}\|_{L^2(\Omega)} &\leq C \Big( \frac{|\im s|}{|\re s|} + 1 \Big) \|b\|_{H^2(S_y(S^{n-1}_+))}
	\end{align*}}
     with a constant $C$ which depends on $M$ and $\Omega$.
\end{thm}
\begin{proof}
In order to construct $u$ satisfying the above properties, we use the computations from subsections \ref{ssec:ccgo1} and \ref{ssec:ccgo2} and show that there exists $\tilde{r}_{s, q}$ solving the equation
\begin{equation*}
|x -x_0|^{\re s} (-\Delta + q) (|x - x_0|^{-\re s} \tilde{r}_{s,q})
= |x -x_0|^{\re s} (\Delta - q) (|x-x_0|^{-s-\frac{n - 2}{2}}v_s).
\end{equation*}
Note that considering $r_{s,q}(x) = |x - x_0|^{i\im s} \tilde{r}_{s,q}(x)$ we have $u$ as in the statement. The existence of $\tilde{r}_{s, q}$ solving this equation and satisfying the corresponding bounds follows from the method%
\footnote{The method of a priori estimates for solving linear equations consists in using a priori estimates, the Hahn-Banach theorem (or simpler extension theorems for a functional) and a Riesz representation theorem.}
of a priori estimates for the adjoint of $|x -x_0|^{\re s} (-\Delta + q) |x - x_0|^{-\re s}$ see \cite{BU} for a similar case. In this case, the required a priori estimate is a Carleman estimate which was stated in \cite[(4.3)]{KSU}.
\end{proof}


%
%
\section{Stability estimates}

In this section, we take $q_1,q_2 \in \mathcal{Q}(M,\sigma)$.

\subsection{Harmonic preliminaries}

A compactly supported function $f$ has an analytic Fourier transform, therefore if $\widehat{f}$ vanishes on a compact interval, the function $f$ itself vanishes identically. We will need a quantitative version of this statement. This is the object of the following lemma. 
\begin{lem}
\label{CGO:HarmLem}
     Let $H$ be a Hilbert space and $\sigma \in (0,1]$,  there exists a constant $C_{\sigma}>0$ such that for all $K > 0$, all $0 < \lambda_0 \leq 1$ and all functions $f \in L^{1}_{\rm comp}(\R;H) \cap H^{\sigma}(\R;H)$ with values in $H$ such that
         $$ \|f\|_{L^{1}}+\|f\|_{H^{\sigma}} \leq K, $$
     we have 
     \begin{align*}
          \|f\|_{L^2(\R;H)} \leq C_{\sigma} \max(1,K)^2 \, \e^{2L\lambda_0} \lambda_0^{-\frac{1}{2}-2\sigma} \,  \bigg|\log \sup_{|\lambda| \leq \lambda_0}  \|\widehat{f}(\lambda)\|_H \,\bigg|^{-\frac{\sigma}{3+2\sigma}}  
     \end{align*}
     where  $\supp f \subset [-L,L]$.
\end{lem}
\begin{rem}
     Note that if $K \geq 1$ one has
     \begin{align*}
          \big|\log \big(M/K^2\big)\big|^{-1} \leq |\log M|^{-1} \quad \text{ when } M \leq K 
     \end{align*}
     and one can change $f$ into $f/\max(1,K)^2$ to reduce to the case $K \leq 1$.
\end{rem}
\begin{proof} 
      After scaling and after changing $f$ into $f/\max(1,K)^2$ we may suppose $\lambda_0=1$ and $K=1$. We denote
          $$ M = \sup_{|\lambda| \leq 1} \|\widehat{f}(\lambda)\|_H \leq  1 $$
          and let $\phi$ be the harmonic function on the upper half-plane $\H$ which decays at infinity and whose restriction to the real line is the characteristic function of the interval $[-1,1]$
     \begin{align*}
          \phi(\lambda +i\mu) &= \frac{1}{\pi}  \int_{-1}^1 \frac{\mu \, \dd m}{(m-\lambda)^2+\mu^2} \\ &= \frac{1}{\pi}\bigg(\arctan\bigg(\frac{1-\lambda}{\mu}\bigg)+\arctan\bigg(\frac{1+\lambda}{\mu}\bigg)\bigg) 
     \end{align*}
     One can further simplify the expression of $\phi$ 
      \begin{align*}
          \phi(\lambda +i\mu) =
          \begin{cases}   \displaystyle \frac{1}{\pi} \arctan \bigg(\frac{2\mu}{\mu^2+\lambda^2-1}\bigg) & \text{ if } \mu^2+\lambda^2>1 \\  \displaystyle 1+\frac{1}{\pi} \arctan \bigg(\frac{2\mu}{\mu^2+\lambda^2-1}\bigg) & \text{ if } \mu^2+\lambda^2<1.
          \end{cases}
     \end{align*}
     Note that $\phi$ has values in the interval $[0,1]$ and that on the line $\mu=1$ we may give a simple expression of $\phi$
      \begin{align*}
          \phi(\lambda +i) &= \frac{1}{\pi} \arctan \bigg(\frac{2}{\lambda^2}\bigg).
     \end{align*} 
      Besides, the Fourier transform $\widehat{f}$ satisfies the bound
     \begin{align*}
           \|\widehat{f}(\lambda+i\mu)\|_H \leq \e^{\mu L} \|f\|_{L^1} \leq \e^{\mu L}, \quad \mu \geq 0.
     \end{align*}          
     For all unit vectors $v \in H$, the function $\log |\langle \widehat{f}(\lambda+i\mu),v\rangle_H|-\mu L$ is subharmonic, and thanks to the previous bound satisfies
        $$ \log |\langle \widehat{f}(\lambda+i\mu),v\rangle_H|-\mu L \leq 0, \quad \lambda+i\mu \in \C, \, \mu \geq 0. $$
        
     Let $\eps>0$, since $\lim_{|\lambda+i\mu|\to \infty}\phi(\lambda+i\mu)=0$ there exists $R$ large enough so that  we have
          $$ (\log M)\phi(\lambda+i\mu) + \eps \geq \frac{\eps}{2} \geq 0, \quad |\lambda+i\mu|=R. $$
    Since we also have
          $$ \log \|\widehat{f}(\lambda)\|_H \leq (\log M) \phi(\lambda) + \eps, \quad \lambda \in [-R,R] $$   
     by the maximum principle on the semi-disk $\bar{D}(0,R) \cap \bar{\H}$, we deduce
           $$ \log |\langle \widehat{f}(\lambda+i\mu),v\rangle_H|-\mu L \leq (\log M) \phi(\lambda+i\mu) + \eps,
     \quad |\lambda+i\mu| \leq R. $$
     The radius $R$ is arbitrarily large, hence the estimate is valid on the whole upper half-plane, and then one can let $\eps$ tend to zero. Finally, one gets
     \begin{align*}
          |\langle \widehat{f}(\lambda+i\mu),v\rangle_H| \leq  \e^{\mu L} M^{\phi(\lambda+i\mu)}, \quad
           \lambda+i\mu \in \bar{\H}
     \end{align*}
     for all unit vectors $v\in H$, and in particular if $\mu=1$
       \begin{align*}
          \|\widehat{f}(\lambda+i)\|_H = \sup_{v \in H,\|v\|_H=1} |\langle \widehat{f}
          (\lambda+i),v\rangle_H| \leq  \e^{L} M^{\frac{1}{\pi}\arctan\frac{2}{\lambda^2}},
          \quad \lambda \in \R.
     \end{align*}
     
     Thus, for $\Lambda_0 \geq 1$ we obtain
     \begin{align*}
          \int_{-\Lambda_0}^{\Lambda_0}  \|\widehat{f}(\lambda+i)\|_H^2 \, \dd \lambda \leq 2\e^{2L} \int_{0}^{\Lambda_0} \e^{\frac{2}{\pi} (\log M) \arctan\frac{2}{\lambda^2}}  \, \dd \lambda
     \end{align*}
     and using the fact that $\log M \leq 0$ and that $t\e^{-t} \leq 1$ for $t \geq 0$
    \begin{align*}
          \int_{-\Lambda_0}^{\Lambda_0}  \|\widehat{f}(\lambda+i)\|_H^2 \, \dd \lambda &\leq  \pi \e^{2L} \bigg(\int_{0}^{\Lambda_0} \frac{1}{\arctan\frac{2}{\lambda^2}}  \, \dd \lambda\bigg) \, |\log M|^{-1} \\ &\leq \pi \e^{2L} \bigg(\underbrace{\int_{0}^{\infty} \frac{ \dd \lambda}{(1+\lambda)^{3+\sigma} \arctan\frac{2}{\lambda^2}}}_{<\infty}  \bigg)  \, (2\Lambda_0)^{3+\sigma} |\log M|^{-1}.
     \end{align*}     
     For the high frequency part of the integral, we have 
     \begin{align*}
              \int_{|\lambda| \geq \Lambda_0}  \|\widehat{f}(\lambda+i)\|_H^2 \, \dd \lambda 
              \leq \Lambda_0^{-2\sigma}  \int_{-\infty}^{\infty}  \lambda^{2\sigma}\|\widehat{f}(\lambda+i)\|_H^2 \, \dd \lambda \leq \Lambda_0^{-2\sigma} \|\e^{\ell}f\|^2_{H^{\sigma}}
     \end{align*}
     and by interpolation
           $$  \|\e^{\ell}f\|_{H^{\sigma}} \leq 2^{\sigma} \e^L \|f\|_{H^{\sigma}} \leq 2^{\sigma} \e^L. $$
     This finally provides the following bound
     \begin{align*}
          \int_{-\infty}^{\infty}  \|\widehat{f}(\lambda+i)\|_H^2 \, \dd \lambda \leq C^2_{\sigma} \e^{2L} \Big(\Lambda_0^{3+\sigma} \, |\log M|^{-1}  +  \Lambda_0^{-2\sigma}  \Big)
     \end{align*}
     and taking the optimal value $\Lambda_0=|\log M|^{\frac{1}{3+3\sigma}} \geq 1$ gives
     \begin{align*}
          \int_{-\infty}^{\infty}  \|\widehat{f}(\lambda+i)\|_H^2 \, \dd \lambda \leq 2C^2_{\sigma} \e^{2L}  |\log M|^{-\frac{2\sigma}{3+3\sigma}}.
     \end{align*}
     By Plancherel 
     \begin{align*}
          \|f\|_{L^2}^2 \leq \e^{2L} \int_{-\infty}^{\infty}  \e^{2\ell} \|f(\ell)\|_H^2 \, \dd \ell = 
          \frac{\e^{2L}}{2\pi} \int_{-\infty}^{\infty}  \|\widehat{f}(\lambda+i)\|_H^2 \, \dd \lambda 
     \end{align*}
     we deduce the desired bound $\|f\|_{L^2} \leq \e^{2L} C_{\sigma}/\sqrt{\pi} |\log M|^{-\sigma/(3+3\sigma)}$. This completes the proof of the lemma.
\end{proof}

\subsection{Controlling the geodesic ray transform}

Let $u_1,u_2$ be two solutions of the Schr\"odinger equations $(-\Delta+q_j)u_j=0$, $j=1,2$, we start from the formula 
\begin{align*}
     \int_{\Omega} (q_1-q_2) u_1 u_2 \, \dd x = \int_{\d \Omega} (\Lambda_{q_1}-\Lambda_{q_2})u_1 u_2 \, \dd S
\end{align*}
and denote   
	$$  q = q_1-q_2. $$
\begin{rem}
     Note that for $0<\sigma<1/2$ we have $q \in H^{\sigma}(\R^n)$
and $\|q\|_{H^{\sigma}} \leq 2M$. 
\end{rem}
For $\delta>0$ small enough, the set
    $$ F_{2\delta}(x_0) = \big\{x \in \d\Omega : \langle x-x_0,\nu \rangle \leq {2}\delta |x-x_0|^2 \big\} \supset F(x_0) $$
is contained in $\tilde{F}$, we consider a cuttoff function $\chi \in C^{\infty}(\d\Omega)$ supported in $F_{2\delta}(x_0)$ with values in the interval $[0,1]$, which equals one 
on $F_{\delta}(x_0)$, and decompose the difference of Dirichlet-to-Neumann maps in the following way
     $$  \Lambda_{q_1}-\Lambda_{q_2} = \chi (\Lambda_{q_1}-\Lambda_{q_2}) + (1-\chi)
       (\Lambda_{q_1}-\Lambda_{q_2}). $$
This provides the following estimate
\begin{align*}
     \bigg|\int q u_1 u_2 \, \dd x\bigg| &\leq \|\Lambda^{\sharp}_{q_1}-\Lambda^{\sharp}_{q_2}\| \,\|u_1\|_{H^{\frac{1}{2}}(\tilde{B})} \|\chi u_2\|_{H^{\frac{1}{2}}(\tilde{F})} \\ &\quad + \bigg|\int_{\d\Omega \setminus F_{\delta}(x_0)} (1-\chi) (\Lambda_{q_1}-\Lambda_{q_2})u_1 u_2 \, \dd S\bigg|.
\end{align*}
The second right-hand side term may be bounded by
\begin{multline*}
      \bigg|\int_{\d\Omega \setminus F_{\delta}(x_0)} (1-\chi)(\Lambda_{q_1}-\Lambda_{q_2})u_1 u_2 \, \dd S\bigg| \\ \leq 
\big\||x-x_0|^{-\tau} (\Lambda_{q_1}-\Lambda_{q_2})u_1\big\|_{L^2(\d\Omega \setminus F_{\delta}(x_0))} \, \big\||x-x_0|^{\tau} u_2\big\|_{L^2(\d \Omega)}       
\end{multline*}
and using Corollary \ref{CGO:CarlDNClry}, we get
\begin{align*}
     \bigg|\int q u_1 u_2 \, \dd x\bigg| &\leq \big\|\Lambda^{\sharp}_{q_1}-\Lambda^{\sharp}_{q_2}\big\| \,\|u_1\|_{H^{\frac{1}{2}}(\tilde{B})} \|\chi u_2\|_{H^{\frac{1}{2}}(\tilde{F})} \\ &\quad + {C_1 \||x-x_0|^{-\tau}(\Lambda_{q_1}-\Lambda_{q_2})u_1\|_{L^2(F(x_0))}} \big\||x-x_0|^{\tau} u_2\big\|_{L^2(\d \Omega)} \\ &\quad + C_1\tau^{-\frac{1}{2}}  \big\||x-x_0|^{-\tau}(q_1-q_2) u_1\big\|_{L^2({\Omega})}  \big\||x-x_0|^{\tau} u_2\big\|_{L^2(\d \Omega)}.
\end{align*}
{
Note that $(\Lambda_{q_1} - \Lambda_{q_2})u_1 = \partial_\nu u_1|_{\partial \Omega} - \partial_\nu v_2|_{\partial \Omega}$ with $v_2$ solving $(-\Delta + q_2) v_2 = 0$ and $v_2|_{\partial \Omega} = u_1|_{\partial \Omega}$. Therefore, $\partial_\nu (u_1 - v_2)|_{\partial \Omega} \in H^\frac{1}{2}(\partial \Omega)$ and
\[\| \partial_\nu (u_1 - v_2)|_{\partial \Omega} \|_{H^\frac{1}{2}(\partial \Omega)} \leq \tilde{C}_1 \big( \| u_1 - v_2 \|_{L^2(\Omega)} + \| \Delta (u_1 - v_2) \|_{L^2(\Omega)} \big).\]
This is consequence of the following lemma which was stated and proven, for example, in \cite[Lemma 1.1]{BU}:
\begin{lem} Let $\Omega$ denote a bounded open subset with boundary $\partial \Omega$ of class $C^2$. Let $u \in L^2(\Omega)$ such that $\Delta u \in L^2(\Omega)$. If $u|_{\partial \Omega} \in H^\frac{3}{2}(\partial \Omega)$, then $u \in H^2(\Omega)$ and $\partial_\nu u|_{\partial \Omega} \in H^\frac{1}{2} (\partial \Omega)$. Moreover,
\[\| \partial_\nu u \|_{H^\frac{1}{2}(\partial \Omega)} \leq C \big( \| u \|_{L^2(\Omega)} + \| \Delta u \|_{L^2(\Omega)} + \| u \|_{H^\frac{3}{2}(\partial \Omega)} \big).\]
\end{lem}

Moreover, since $(-\Delta + q_2)(u_1 - v_2) = (q_2 - q_1)u_1$ we have that
\[\| \partial_\nu (u_1 - v_2)|_{\partial \Omega} \|_{H^\frac{1}{2}(\partial \Omega)} \leq C'_1 \| u_1 \|_{L^2(\Omega)}.\]
Thus, we can get the following bound
\begin{align*}
	\||x-x_0|^{-\tau}(\Lambda_{q_1}&-\Lambda_{q_2})u_1\|_{L^2(F(x_0))} \\ 
	&\leq \varepsilon^{-\tau} 	\big( \int_{\partial \Omega} \chi (\Lambda_{q_1} - \Lambda_{q_2})u_1 \overline{\partial_\nu (u_1 - v_2)} \dd S \big)^\frac{1}{2}\\
	&\leq (C'_1)^\frac{1}{2} \varepsilon^{-\tau} \big\|\Lambda^{\sharp}_{q_1}-\Lambda^{\sharp}_{q_2}\big\|^\frac{1}{2} \| u_1 \|^\frac{1}{2}_{H^\frac{1}{2}(\partial \Omega)} \| u_1 \|^\frac{1}{2}_{L^2(\Omega)}.
\end{align*}
}
We now choose $u_1,u_2$ to be the solutions constructed in section \ref{sec:CGO}    
\begin{align*}
     u_1{(x)} &= |x-x_0|^{s}\big({ |x - x_0|^{-\frac{n - 2}{2}}}v_s^{(1)}{ (x)}+r_{s,q_1}(x)\big), \\
     u_2{ (x)} &= |x-x_0|^{-\bar{s}}\big({ |x - x_0|^{-\frac{n - 2}{2}}}v^{(2)}_{-\bar{s}}{ (x)}+r_{-\bar{s},q_2}(x)\big)
\end{align*}
with $s=\tau+i\lambda \in \H$ with $\tau \geq \max(1,\tau_0)$. This gives
\begin{align*}
     \bigg|\int q u_1 u_2 \, \dd x\bigg| &\leq C_2 \Big({\varrho_0}^{\tau} \eps^{-\tau} \big\|\Lambda^{\sharp}_{q_1}-\Lambda^{\sharp}_{q_2}\big\|^{\frac{1}{2}}   
       +\tau^{-\frac{1}{2}}\Big) \|b\|_{{H}^2(S_y(S^{n-{1}}_+))}. 
\end{align*}
the left-hand side term of which reads 
\begin{multline}
\label{Stab:MainTerm}
     \int q u_1 u_2 \, \dd x = \int |x-x_0|^{2i\lambda}  q v_s^{(1)}v^{(2)}_{-\bar{s}} { |x - x_0|^{-n+2}}\, \dd x \\ + \int |x-x_0|^{2i\lambda}  q \big({ |x - x_0|^{-\frac{n - 2}{2}}} v_s^{(1)}r_{-\bar{s},q_2} + { |x - x_0|^{-\frac{n - 2}{2}}} v^{(2)}_{-\bar{s}}r_{s,q_1} + r_{s,q_1}r_{-\bar{s},q_2}\big) \, \dd x.
\end{multline}
The first term in the former equality equals
\begin{align*}
     \int_{-\infty}^{\infty}\int_{S^{n-1}_+} \e^{2i\lambda t} \e^{2t} q(x_0+\e^t \omega) v_s^{(1)}v^{(2)}_{-\bar{s}} \, \dd \omega \, \dd t = \int_{S^{n-1}_+} Q(\lambda,\omega) v_s^{(1)}v^{(2)}_{-\bar{s}} \, \dd \omega 
\end{align*}
where $Q$ is {(almost)} the Fourier transform 
\begin{align*}
    Q(\lambda,\omega) = \int_{-\infty}^{\infty} \e^{i 2\lambda t} \e^{2t} q(x_0+\e^t \omega) \, \dd t.
\end{align*}
Note that because of the discussion carried in the beginning of section \ref{sec:CGO} this function is supported in a cap of the hypersphere strictly smaller than the hemisphere
     $$ \supp Q(\lambda,\cdot) \subset \Gamma \subset S^{n-1}_{>\alpha_0}. $$
The second term in \eqref{Stab:MainTerm} may be bounded by  
\begin{align*}
     2M \big( \|v_s^{(1)}\|_{L^2(\Omega)} \|r_{-\bar{s},q_2}\|_{L^2(\Omega)} + \|v^{(2)}_{-\bar{s}}\|_{L^2(\Omega)} \|r_{s,q_1}\|_{L^2(\Omega)} \\ + \|r_{s,q_1}\|_{L^2(\Omega)} \|r_{-\bar{s},q_2}\|_{L^2(\Omega)}\big) \leq C_3 \tau^{-1} \|b\|_{{H}^2(S_y(S^{n-{1}}_+))}.
\end{align*}
\begin{rem}
     Note that the difference with respect to the Kenig, Sj\"ostrand and Uhlmann approach \cite{KSU} used in \cite{CDSFR} is that $x_0$ is \emph{fixed}.
\end{rem}
We finally get
\begin{multline}
\label{Stab:FirstEst}
     \bigg|\int_{S^{n-1}_+} Q(\lambda,\omega) v_s^{(1)}v^{(2)}_{-\bar{s}} \, \dd \omega\bigg| \\ \leq  C_4 \Big({\varrho_0}^{\tau} \eps^{-\tau} \big\|\Lambda^{\sharp}_{q_1}-\Lambda^{\sharp}_{q_2}\big\|^{\frac{1}{2}}   
       +\tau^{-\frac{1}{2}}\Big) \|b\|_{{H}^2(S_y(S^{n-{1}}_+))}. 
\end{multline}
Starting from the formula
\begin{align*}
     \int_{\Omega} (q_1-q_2) u_1 u_2 \, \dd x = \int_{\d \Omega} u_1(\Lambda_{q_1}-\Lambda_{q_2})u_2 \, \dd S
\end{align*}
and choosing the solutions to be 
\begin{align*}
     u_1{ (x)} &= |x-x_0|^{-s}\big({ |x - x_0|^{-\frac{n - 2}{2}}}v_s^{(1)}{ (x)}+r_{s,q_1}(x)\big), \\
     u_2{ (x)} &= |x-x_0|^{\bar{s}}\big({ |x - x_0|^{-\frac{n - 2}{2}}}v^{(2)}_{-\bar{s}}{ (x)}+r_{-\bar{s},q_2}(x)\big)
\end{align*}
with $s=\tau+i\lambda \in \H$, we get similar estimates as \eqref{Stab:FirstEst} on 
\begin{align*}
      \int_{S^{n-1}_+} Q(-\lambda,\omega) v_s^{(1)}v^{(2)}_{-\bar{s}} \, \dd \omega.
\end{align*}
With the choice
\begin{align*}
     v_s^{(1)}&= \e^{i s\theta} (\sin \theta)^{-\frac{n-2}{2}} b(\eta)  \\
     v^{(2)}_{-\bar{s}}&= \e^{-i \bar{s}\theta} (\sin \theta)^{-\frac{n-2}{2}}, \qquad \omega = \exp_{y}(\theta \eta)
\end{align*}
the left-hand side term in the estimate \eqref{Stab:FirstEst} relates to the attenuated geodesic ray transform on the hemisphere
\begin{align*}
     \int_{S^{n-1}_+} Q(\lambda,\omega) v_s^{(1)}v^{(2)}_{-\bar{s}} \, \dd \omega  
     = \int_{S^{n-2}_+} \bigg(\underbrace{\int_0^{\pi}  Q(\lambda,\exp_y(\theta \eta)) \e^{-{2}\lambda \theta}  \, \dd\theta}_{=T_{{2}\lambda}(Q(\lambda,\cdot))(y,\eta) }\bigg) b(\eta) \, \dd \eta. 
\end{align*}
If $\big\|\Lambda^{\sharp}_{q_1}-\Lambda^{\sharp}_{q_2}\big\|$ is small enough, we choose 
     $$ \tau= \frac{1}{{4} (\log {\varrho_0} -\log \eps)} { \Big|} \log \big\|\Lambda^{\sharp}_{q_1}-\Lambda^{\sharp}_{q_2}\big\| { \Big|}  \geq \tau_0 $$ 
and we get the following estimate. 
\begin{prop}
     There exist two constants $\kappa$ and $C>0$ such that if $\big\|\Lambda^{\sharp}_{q_1}-\Lambda^{\sharp}_{q_2}\big\| \leq \kappa$  then
     \begin{align*}
     \bigg|\int_{S^{n-2}_+} T_{{2}\lambda}(Q(\lambda,\cdot))(y,\eta) b(\eta) \, \dd \eta\bigg| \leq C \big| \log\big\|\Lambda^{\sharp}_{q_1}-\Lambda^{\sharp}_{q_2}\big\|\big|^{-\frac{1}{2}} \, \|b\|_{H^2(S_y(S^{n-1}_+))}
\end{align*}
     for all {$y \in \partial S^{n-1}_+$} and all smooth functions $b$ on $S_y(S^{n - 1}_+)$ {which has been identified with $S^{n-{2}}_+$}.
\end{prop}

\subsection{Stability estimates}
From the proposition, we get the norm estimate
\begin{align*}
\sup_{y\in \partial S^{n-1}_+}
     \|T_{{2}\lambda}(Q(\lambda,\cdot))\|_{H^{-2}(S^{n-{2}}_y)} \leq C_1 \big| \log\big\|\Lambda^{\sharp}_{q_1}-\Lambda^{\sharp}_{q_2}\big\|\big|^{-\frac{1}{2}}.
\end{align*}
By interpolation, we have
\begin{multline*}
      \|T_{{2}\lambda}(Q(\lambda,\cdot))\|_{L^2(\partial S^{n-1}; L^2(S^{n-{2}}_y))} \\  \leq
       \|T_{{2}\lambda}(Q(\lambda,\cdot))\|_{L^2(\partial S^{n-1}_+;H^{-2}(S^{n-{2}}_y))}^{\frac{\sigma}{\sigma+2}} \|T_{{2}\lambda}(Q(\lambda,\cdot))\|_{L^2(\partial S^{n-1}_+;H^{\sigma}(S^{n-{2}}_y))}^{\frac{2}{\sigma+2}}
\end{multline*}
and by Remark \ref{Xray:SobolevContinuityRem} we have the bound
\begin{align*}
     \|T_{{2}\lambda}(Q(\lambda,\cdot))\|_{L^2(\partial S^{n-1}_+;H^{\sigma}({ S^{n-2}_y}))} \leq C_2 \|Q(\lambda,\cdot))\|_{H^{\sigma}(S^{n-1})}.\end{align*}
Since $q\in L^2(\R; H^\sigma(S^{n-1}) )$ from the compactness of the support it follows that
$q\in L^1(\R; H^\sigma(S^{n-1}) )$ and from vector valued Riemann-Lebesgue
$$ \sup_{\lambda}\| Q(\lambda, \cdot)\|_{ H^\sigma(S^{n-1}) }\leq C_3M.$$
Hence
$$\|T_{{2}\lambda}(Q(\lambda,\cdot))\|_{L^2(E)} \leq C_4 \big| \log\big\|\Lambda^{\sharp}_{q_1}-\Lambda^{\sharp}_{q_2}\big\|\big|^{-\frac{\sigma}{2(\sigma+2)}}$$
then using the estimate on the attenuated geodesic ray transform of Theorem \ref{Xray:AttXraySpherEst} in section \ref{Xray:sec} we deduce 
\begin{align*}
     \|Q(\lambda,\cdot)\|_{H^{-\frac{1}{2}}(S^{n-1}_+)} \leq C_5 \big| \log\big\|\Lambda^{\sharp}_{q_1}-\Lambda^{\sharp}_{q_2}\big\|\big|^{-\frac{\sigma}{2(\sigma+2)}}
\end{align*}
for $|\lambda| \leq \lambda_0$ for a small enough $\lambda_0$.

Again  interpolation for  fix $\lambda<\lambda_0 $ with
$$ \sup_{\lambda}\| Q(\lambda, \cdot)\|_{ H^\sigma(S^{n-1}) }\leq C_3M$$
provides the bound
\begin{align*}
    \sup_{|\lambda| \leq \lambda_0} \|Q(\lambda,\cdot)\|_{L^2(S^{n-1}_+ )} \leq
    C_6 \big| \log\big\|\Lambda^{\sharp}_{q_1}-\Lambda^{\sharp}_{q_2}\big\|\big|^{-\frac{\sigma^2}{(2\sigma+1)(\sigma+2)}}
\end{align*}
and using the fact that
\begin{align*}
     \int_{\Omega} |q(x)|^2 |x-x_0|^{-n+4} \, \dd x &=  \int_{-\infty}^{\infty}\int_{S^{n-1}_+} |\e^{2t}q(x_0+\e^t \omega)|^2 \, \dd \omega \, \dd t
\end{align*}
together with Lemma \ref{CGO:HarmLem} with $H=L^2(S^{n-1}_+)$ as a Hilbert space, recalling that $Q(\lambda,\cdot)$ is the following Fourier transform
\begin{align*}
    Q(\lambda,\omega) = \int_{-\infty}^{\infty} \e^{i 2\lambda t} \e^{2t} q(x_0+\e^t \omega) \, \dd t
\end{align*}
we end with the estimate
\begin{multline*}
     \bigg(\int_{\Omega} |q(x)|^2 |x-x_0|^{-n+4} \, \dd x\bigg)^{\frac{1}{2}} \\ \leq  
     C_7 \bigg|\log \bigg(\frac{\sigma^2}{(2\sigma+1)(\sigma+2)} \Big|\log\big\|\Lambda^{\sharp}_{q_1}-\Lambda^{\sharp}_{q_2}\big\|\Big|\bigg) \bigg|^{-\frac{2\sigma}{3+3\sigma}}
\end{multline*}
with a constant $C_7$ depending on $\sigma,M$. If we assume $\big\|\Lambda^{\sharp}_{q_1}-\Lambda^{\sharp}_{q_2}\big\| \leq \kappa'$ with $\kappa' \leq \kappa$ small enough so that
       $$ |\log \kappa| \geq 2 \bigg|\log  \frac{\sigma^2}{(2\sigma+1)(\sigma+2)}\bigg| $$
then we get
\begin{align*}
     \bigg(\int_{\Omega} |q(x)|^2 |x-x_0|^{-n+4} \, \dd x\bigg)^{\frac{1}{2}} \leq  
     C_8 \bigg|\log \Big|\log\big\|\Lambda^{\sharp}_{q_1}-\Lambda^{\sharp}_{q_2}\big\|\Big| \bigg|^{-\frac{2\sigma}{3+3\sigma}}.
\end{align*}

This proves Theorem \ref{Intro:MainThm} when $\big\|\Lambda^{\sharp}_{q_1}-\Lambda^{\sharp}_{q_2}\big\| \leq \kappa'$ once one has observed that
\begin{align*}
   \|q\|_{L^2(\Omega)} \leq C \bigg(\int_{\Omega} |q(x)|^2 |x-x_0|^{-n+4} \, \dd x\bigg)^{\frac{1}{2}}
\end{align*}
with a constant $C$ which depends on the dimension
\begin{align*}
     C=
     \begin{cases}
          \varrho_0 & \text{ if } n=3 \\ 1 & \text{ if } n=4 \\ \eps^{-n+4} & \text{ if } n \geq 5
     \end{cases}.
\end{align*}
When $\big\|\Lambda^{\sharp}_{q_1}-\Lambda^{\sharp}_{q_2}\big\| \geq \kappa$, the bound in Theorem \ref{Intro:MainThm} is self evident. 
%
%
\section{The conductivity equation}

The key observation to relate the Schr\"odinger and the conductivity equations is the following algebraic formula
\begin{align}
     \mathop{\rm div}(\gamma \nabla u) = \sqrt{\gamma}(\Delta-q)v,
     \quad v=\sqrt{\gamma} u, \quad q=\frac{\Delta \sqrt{\gamma}}{\sqrt{\gamma}}.
\end{align}
This leads to the well known relation {\cite{SyU}} linking the Dirichlet-to-Neumann map associated to the conductivity equation with the one associated to the Schr\"odinger equation with potential $q=\Delta\sqrt{\gamma}/\sqrt{\gamma}$
\begin{align}
      \Lambda_q f = \frac{1}{\sqrt{\gamma}} \, \Lambda_\gamma  \, \frac{1}{\sqrt{\gamma}} + \frac{1}{2} \d_{\nu}(\log \gamma) \, {\rm
      id}_{\d\Omega}. \label{for:conduc-Schr_DtM}
\end{align}
{
\begin{lem} There exists a constant $C$ such that
\[\big\| \Lambda_{q_1}^\sharp - \Lambda_{q_2}^\sharp \big\| \leq C \big( \big\| \Lambda_{\gamma_1}^\sharp - \Lambda_{\gamma_2}^\sharp \big\| + \| \gamma_1 - \gamma_2 \|_{L^\infty(\partial \Omega)} + \| \nabla \gamma_1 - \nabla \gamma_2 \|_{L^\infty(\partial \Omega)} \big).
 \]
\end{lem}
\begin{proof}
Using the formula \eqref{for:conduc-Schr_DtM} one can get the identity
\begin{align*}
	\Lambda_{q_1} - \Lambda_{q_2} =& \Big(\frac{1}{\sqrt{\gamma_1}} - \frac{1}{\sqrt{\gamma_2}}\Big) \Lambda_{\gamma_1} \frac{1}{\sqrt{\gamma_1}} + \frac{1}{\sqrt{\gamma_2}} \Lambda_{\gamma_1} \Big(\frac{1}{\sqrt{\gamma_1}} - \frac{1}{\sqrt{\gamma_2}}\Big)\\
	&+ \frac{1}{\sqrt{\gamma_2}} (\Lambda_{\gamma_1} - \Lambda_{\gamma_2}) \frac{1}{\sqrt{\gamma_2}} + \frac{1}{2} \Big( \frac{\partial_\nu \gamma_1}{\gamma_1} - \frac{\partial_\nu \gamma_2}{\gamma_2} \Big)\, {\rm
      id}_{\d\Omega}.
\end{align*}
Now the estimate stated in the lemma follows immediately.
\end{proof}
\begin{lem} There exists a constant $C$ such that
\[ \| \gamma_1 - \gamma_2 \|_{H^1(\Omega)} \leq C \big( \| q_1 - q_2 \|_{L^2(\Omega)} + \| \gamma_1 - \gamma_2 \|_{L^\infty(\partial \Omega)} + \| \nabla \gamma_1 - \nabla \gamma_2 \|_{L^\infty(\partial \Omega)} \big) \]
\end{lem}
\begin{proof}
In order to prove the estimate in the statement we only need to check that $u = \log \sqrt{\gamma_1} - \log \sqrt{\gamma_2}$ is the unique solution for the elliptic equation
\[\nabla \cdot (\sqrt{\gamma_1 \gamma_2} \nabla u) = \sqrt{\gamma_1 \gamma_2} (q_1 - q_2). \]
Thus,
\[\| u \|_{H^1(\Omega)} \leq C ( \| q_1 - q_2 \|_{L^2(\Omega)} + \| u|_{\partial \Omega} \|_{H^\frac{1}{2}(\partial \Omega)}),\]
which implies the estimate in the lemma.
\end{proof}
>From the previous lemmas and the stability of the inverse problem for the Schr\"odinger equation with partial data one gets
\begin{align*}
\| \gamma_1 - \gamma_2 \|_{H^1(\Omega)} \leq C \bigg|\log \Big|\log \Big(& \big\|\Lambda^{\sharp}_{\gamma_1}-\Lambda^{\sharp}_{\gamma_2} \big\|
+ \| \gamma_1 - \gamma_2 \|_{L^\infty(\partial \Omega)}\\
 & + \| \nabla \gamma_1 - \nabla \gamma_2 \|_{L^\infty(\partial \Omega)} \Big)\Big| \bigg|^{-\frac{2\sigma}{3+3\sigma}}.
\end{align*}
This concludes the proof of stability for the Calder\'on problem with partial data.}
\end{document}